\documentclass[11pt]{article}
\usepackage[margin=1in]{geometry}
\usepackage{algorithm}
\usepackage{algpseudocode}
\usepackage{url}

\providecommand{\headers}[2]{}

\let\oldtitle\title
\renewcommand{\title}[2][]{\oldtitle{#2}}

\makeatletter
\newcommand{\LoadPackageUnlessAcm}[2][]{%
    \@ifclassloaded{acmart}{%
        \PackageInfo{preamble}{Skipping package `#2' because acmart class is used}%
    }{%
        \if\relax\detokenize{#1}\relax
            \usepackage{#2}%
        \else
            \usepackage[#1]{#2}%
        \fi
    }%
}
\makeatother

\LoadPackageUnlessAcm{caption}
\LoadPackageUnlessAcm{subcaption}
\LoadPackageUnlessAcm{authblk}
\LoadPackageUnlessAcm{hyperref}

\usepackage{amstext}
\usepackage{stmaryrd}
\usepackage{algpseudocode}

\usepackage{units}
\usepackage{cancel}
\usepackage{graphicx}

\usepackage{xcolor}
\usepackage{array}
\usepackage{verbatim}
\usepackage{booktabs}
\usepackage{makecell}
\usepackage{calc}
\usepackage{textcomp}
\usepackage{multirow}
\usepackage{tikz,tikzscale}

\usepackage{pgfplots}
\usepackage{pgfplotstable}
\usepackage{pgfkeys}
\pgfplotsset{compat=1.18}

\definecolor{deepgreen}{RGB}{0,100,0}

\newif\ifrunscripts
\runscriptstrue

\pgfplotsset{
    p2/.style={teal, mark=*},
    p3/.style={orange, mark=triangle*},
    p7/.style={blue, mark=square*},
}

\newcommand{\plotwidth}{0.49\columnwidth}
\newcommand{\plotheight}{0.4\columnwidth}

\pgfplotscreateplotcyclelist{paper markers}{
    {teal,    solid, mark=*},
    {orange,   solid, mark=triangle*},
    {blue,     solid, mark=square*},
    {red,      solid, mark=pentagon*},
    {violet,     solid, mark=otimes*},
    {brown,    solid, mark=halfcircle*},
    {black,   solid, mark=diamond*},
}

\colorlet{BarColorOne}{orange!60!white}
\colorlet{BarColorTwo}{teal!80!white}
\colorlet{BarColorThree}{red!80!black}
\colorlet{BarColorFour}{blue!60!white}
\colorlet{BarColorFive}{brown!70!black}

\pgfplotsset{legend swatch/.style={area legend, draw=none}}

\pgfplotsset{
    paperplot/.style={
            width=\plotwidth,
            height=\plotheight,
            cycle list name=paper markers,
            every axis/.append style={font=\small},
            grid=major,
            thick
        }
}

\pgfplotsset{
    longplot/.style={
            width=0.8\columnwidth,
            height=\plotheight,
            cycle list name=paper markers,
            every axis/.append style={font=\small},
            grid=major,
            thick
        }
}

\usepackage[all]{xy}
\usepackage{mathtools}
\usepackage{placeins}
\usetikzlibrary{calc}
\usetikzlibrary{arrows.meta}
\usetikzlibrary{shadings}
\usepgfplotslibrary{fillbetween}

\usepackage[mode=multiuser,status=draft]{fixme}
\fxsetup{innerlayout=inline}
\fxsetup{targetlayout=colorcb}
\fxusetheme{color}
\FXRegisterAuthor{mw}{envmw}{MW}

\definecolor{apply}{rgb}{0.3,.7,0.}
\definecolor{applyface}{rgb}{.9,.95,0.}
\definecolor{invert}{rgb} {0.5,0.,1.}
\colorlet{invertface}{invert!60!red}

\usepackage[colorinlistoftodos,prependcaption,textsize=small]{todonotes}

\def\R{\mathbb R}

\usepackage{colortbl}

\usepackage{booktabs}
\usepackage[font=footnotesize,position=b]{subcaption}
\usepackage{adjustbox}

\makeatletter
\def\pgfplots@getautoplotspec into#1{%
        \begingroup
        \let#1=\pgfutil@empty
        \pgfkeysgetvalue{/pgfplots/cycle multi list/@dim}\pgfplots@cycle@dim
        \let\pgfplots@listindex=\pgfplots@numplots
        \pgfkeysgetvalue{/pgfplots/cycle list set}\pgfplots@listindex@set
        \ifx\pgfplots@listindex@set\pgfutil@empty
        \else
            \c@pgf@counta=\pgfplots@listindex
            \c@pgf@countb=\pgfplots@listindex@set
            \advance\c@pgf@countb by -\c@pgf@counta
            \globaldefs=1\relax
            \edef\setshift{%
                \noexpand\pgfkeys{
                    /pgfplots/cycle list shift=\the\c@pgf@countb,
                    /pgfplots/cycle list set=
                }
            }%
            \setshift%
            \globaldefs=0\relax
        \fi
        \pgfkeysgetvalue{/pgfplots/cycle list shift}\pgfplots@listindex@shift
        \ifx\pgfplots@listindex@shift\pgfutil@empty
        \else
            \c@pgf@counta=\pgfplots@listindex\relax
            \advance\c@pgf@counta by\pgfplots@listindex@shift\relax
            \ifnum\c@pgf@counta<0
                \c@pgf@counta=-\c@pgf@counta
            \fi
            \edef\pgfplots@listindex{\the\c@pgf@counta}%
        \fi
        \ifnum\pgfplots@cycle@dim>0
            \c@pgf@counta=\pgfplots@cycle@dim\relax
            \c@pgf@countb=\pgfplots@listindex\relax
            \advance\c@pgf@counta by-1
            \pgfplotsloop{%
                \ifnum\c@pgf@counta<0
                    \pgfplotsloopcontinuefalse
                \else
                    \pgfplotsloopcontinuetrue
                \fi
            }{%
                \pgfkeysgetvalue{/pgfplots/cycle multi list/@N\the\c@pgf@counta}\pgfplots@cycle@N
                \pgfplotsmathmodint{\c@pgf@countb}{\pgfplots@cycle@N}%
                \divide\c@pgf@countb by \pgfplots@cycle@N\relax
                \expandafter\pgfplots@getautoplotspec@
                \csname pgfp@cyclist@/pgfplots/cycle multi list/@list\the\c@pgf@counta @\endcsname
                {\pgfplots@cycle@N}%
                {\pgfmathresult}%
                \t@pgfplots@toka=\expandafter{#1,}%
                \t@pgfplots@tokb=\expandafter{\pgfplotsretval}%
                \edef#1{\the\t@pgfplots@toka\the\t@pgfplots@tokb}%
                \advance\c@pgf@counta by-1
            }%
        \else
            \pgfplotslistsize\autoplotspeclist\to\c@pgf@countd

            \pgfplots@getautoplotspec@{\autoplotspeclist}{\c@pgf@countd}{\pgfplots@listindex}%
            \let#1=\pgfplotsretval
        \fi
        \pgfmath@smuggleone#1%
        \endgroup
    }

\pgfplotsset{
    cycle list set/.initial=
}
\makeatother

\usepackage{lipsum}
\usepackage{amsfonts}
\usepackage{graphicx}
\usepackage{epstopdf}
\usepackage{amsmath,amssymb,amsthm}
\newtheorem{lemma}{Lemma}
\newtheorem{theorem}{Theorem}

\theoremstyle{definition}

\theoremstyle{remark}
\newtheorem{remark}{Remark}
\theoremstyle{plain}
\ifpdf
    \DeclareGraphicsExtensions{.eps,.pdf,.png,.jpg}
\else
    \DeclareGraphicsExtensions{.eps}
\fi

\headers{Matrix-Free Finite Elements in Cell-Wise Storage}{Wichrowski}

\title{Coalesced Matrix-Free Finite Elements in Cell-Wise Storage}
\author{Michał Wichrowski}

\usepackage{amsopn}

\providecommand{\keywordsname}{Keywords}
\newcommand{\Keywords}[1]{\par\noindent\textbf{\keywordsname:} #1}
\newcommand{\AMS}[1]{\par\noindent\textbf{AMS subject classifications:} #1}

\begin{document}
\author{Michał Wichrowski}
\footnotetext{Interdisziplinäres Zentrum für Wissenschaftliches Rechnen (IWR), Ruprecht-Karls-Universität Heidelberg, Germany, \texttt{mwichro@mimuw.edu.pl}}
\date{}

\maketitle

\begin{abstract}
    We present a GPU-oriented formulation of continuous high-order finite elements in which the redundant, cell-wise
(element-local) vector is the persistent primary representation of all field data, rather than a transient stage of
matrix-free operator evaluation. We prove that, given a preconditioner whose image is continuous, the entire flexible
conjugate gradient iteration can be carried out exactly on this unassembled representation: a simple primal-dual
pairing identity shows that all Krylov scalars computed from local data coincide with those of the assembled solve, so
inter-element communication is confined entirely to the preconditioner. The required direct stiffness summation (DSS)
is then realized without indirect gather-scatter, atomics, or coloring, by a dimensionally-split cascade of one-to-one
face exchanges that provably accumulates edge and vertex contributions as a byproduct of sequential axis passes;
unstructured macro-block interfaces and $h$-adaptive hanging nodes are handled by disjoint topological kernels and a
shadow-cell wrapper that leaves the high-throughput sweeps untouched. The cell-wise storage decouples the memory layout
from the mesh topology, and we exploit this freedom to benchmark blocked layouts that trade memory coalescing against
element contiguity. Numerical experiments on modern GPUs demonstrate that the resulting operator evaluation and solver
outperform state-of-the-art matrix-free implementations, signifficantly exceeding throughput of existing
implementations.

\end{abstract}


\AMS{65Y10, 65Y20, 65N55, 65N30}

\section{Introduction}
\label{sec:intro}

Matrix-free evaluation of high-order finite element operators has become the method of choice for large-scale
simulation on modern hardware. Instead of assembling and storing a sparse matrix, the operator action is recomputed on
the fly from element-local data using sum factorization, reducing both memory footprint and memory traffic by an order
of magnitude at moderate polynomial degrees~\cite{Orszag1980,deville2002highorder,Kronbichler2012,bastian2019matrix}.
This shift is driven by the hardware itself: over the past two decades the floating-point throughput of processors has
grown far faster than their memory bandwidth, so that virtually all sparse and element-wise operations are limited by
data movement rather than arithmetic~\cite{williams2009roofline}. The trend is most pronounced on GPUs, which now
dominate the upper end of the TOP500 list and deliver their nominal bandwidth only under stringent conditions on the
access pattern.

GPUs add two further twists. First, their memory subsystem is throughput-oriented: peak bandwidth is attained only when
the threads of a warp access contiguous, aligned memory (coalescing), and any indirection-based access pattern forfeits
a substantial fraction of it. Second, an ever-growing share of the silicon is devoted to matrix units (tensor cores)
and to reduced-precision arithmetic, to the point that double-precision capability stagnates or is emulated through
lower-precision units~\cite{mukunoki2025dgemm}. An algorithm aiming to maximize hardware utilization should therefore
(i) arrange its data so that every load and store is coalesced, (ii) cast its inner kernels as small dense matrix
products, and (iii) tolerate reduced or mixed precision inside the preconditioner~\cite{goddeke2007performance}. These
three requirements shape the design presented in this paper.

State-of-the-art matrix-free frameworks: deal.II~\cite{Kronbichler2012,kronbichler2022cg,dealII97},
MFEM~\cite{mfem2021}, libCEED~\cite{libceed2021, ABDELFATTAH2021102841}, Nek5000/NekRS~\cite{nekrs2022,
    offermans2016strong}, and atmospheric models like NUMA~\cite{muller2019strong} organize the operator action around an
assembled global vector of unique degrees of freedom. Every operator application then begins by gathering element-local
copies through an index map and ends by scattering the element contributions back, with atomics or mesh coloring to
serialize concurrent writes. On a GPU this transition is precisely where the bandwidth is lost: the indirect,
warp-divergent accesses of the gather-scatter stage are the dominant cost of the otherwise highly tuned kernels, a fact
documented consistently across CEED benchmark studies~\cite{fischer2020scalability,kronbichler2022cg} and petascale
strong-scaling analyses~\cite{offermans2016strong, muller2019strong}. The arithmetic of sum factorization has been
polished for a decade; the remaining slack lies in how the data is stored and moved.

In this paper, we adopt an alternative perspective: we make the redundant, cell-wise (element-local) vector the
\emph{persistent} primary representation of all field data, and we reorganize the Krylov solver so that it operates on
this representation \emph{directly}. The assembled vector is never formed; continuity enters only through a direct
stiffness summation (DSS) operator applied inside the preconditioner. All storage becomes block-structured with
statically known offsets, so every kernel (operator application, vector updates, and DSS itself) reads and writes
coalesced, indirection-free memory.

The individual ingredients of our approach (element-local cell-wise storage, sum-factorized matrix-free kernels, and
direct stiffness summation) are well established. Element-local redundant vectors have been a staple of the
spectral-element community since the work of Deville, Fischer, and Mund~\cite{deville2002highorder}; the libCEED
library~\cite{libceed2021} formalizes the very decomposition we build on through its L-vector/E-vector/Q-vector
hierarchy, and Nek5000/NekRS~\cite{nekrs2022} store fields element-locally and invoke gather-scatter
(\texttt{gslib}~\cite{gslib, gslib2}) whenever continuity is required. The decisive difference is that these frameworks
treat the element-local vector as transient kernel input and output, while we make it the solver's persistent state:
the assembled representation is never formed at all, and all vector algebra (updates, duality pairings, and norms) is
performed directly on the redundant representation.

Our first contribution is the theoretical framework that justifies this design. The key duality-pairing identity,
stating that pairing an unassembled dual vector with a continuous primal vector reproduces the assembled inner product,
is mathematically elementary and has long been exploited in parallel spectral-element codes and in domain-decomposition
methods of FETI/BDDC type~\cite{farhat1991feti,dohrmann2003bddc}, where inner products between one assembled and one
unassembled vector are standard. What we have not found stated and proved in this form is the resulting equivalence
theorem: provided the preconditioner maps into the continuous subspace, the \emph{entire} flexible conjugate gradient
iteration~\cite{notay2000flexible} executed on unassembled cell-wise data is identical, iterate by iterate, tostandard
(Flexible) CG applied to the assembled global system. Assembly is thereby provably required exactly once per iteration,
inside the preconditioner, and nowhere else. While standard formulations of the DSS operator rely on mapping local
element-wise data to an assembled global vector~\cite[Ch.~7]{giraldo2020}, we bypass this requirement entirely by
operating strictly on the persistent E-vector.

We note that communication is not eliminated but \emph{confined}: the DSS sweep inside the smoother replaces all
indirect gather-scatter traffic of the classical formulation. We also emphasize that DSS is merely the realization of
this requirement adopted in this paper, not an ingredient of the theory: the equivalence theorem only demands that the
preconditioner consume the residual through its assembly and return a continuous (more generally,
constraint-satisfying) vector. Any preconditioner with this property qualifies, in particular overlapping patch
smoothers of additive or multiplicative Schwarz
type~\cite{pavarino1993additive,ArnoldFalkWinther00,fischer2000overlapping,WitteArndtKanschat21,brubeck2021scalable,cui2025implementation,wichrowski2025local},
whose patch-local solves produce continuous corrections by construction.

Our second contribution, and the algorithmic core of the paper, is a DSS operator that requires neither indirect
addressing nor atomic operations. Existing GPU gather-scatter implementations, such as \texttt{gslib} on GPUs and the
element restrictions of libCEED and MFEM~\cite{mfem2021}, traverse index lists, typically combined with atomics or mesh
coloring. We instead exploit the block-structured cell-wise layout, in the spirit of the statically addressed
macro-block storage of hierarchical hybrid grids~\cite{bergen2004hierarchical,gmeiner2015towards}: a
dimensionally-split cascade of three axis-wise passes, each consisting of one-to-one face exchanges with statically
known offsets, provably accumulates the four-way edge and eight-way vertex sums as a byproduct of the sequential
sweeps. Unstructured macro-block interfaces are handled by mutually disjoint vertex, line, and face kernels with
on-the-fly integer-arithmetic orientation canonicalization. A shadow-cell wrapper that extends the same machinery to
$h$-adaptive hanging nodes (reducing the non-conforming case exactly to the conforming one so that adaptivity never
touches the high-throughput sweeps) is described in Appendix~\ref{sec:dss_hanging}; it is not required by the solver,
since the companion multigrid paper~\cite{wichrowski2026MG} shows that DSS across hanging nodes can be avoided
altogether.

Our third contribution is an engineering study enabled by the storage paradigm itself. Because cell-wise storage
decouples the memory layout from the global mesh topology, the layout becomes a free tuning parameter rather than a
consequence of the assembly data structures. Blocked and interleaved element layouts have been explored before, for
instance via vectorization over cells in deal.II~\cite{Kronbichler2012,kronbichler2022cg} and struct-of-arrays element
batching in MFEM and libParanumal. However, the persistent cell-wise vector lets us benchmark the full design space,
including the blocked tiling that trades memory coalescing against element contiguity, under identical solver
conditions. We quantify the resulting bandwidth utilization and demonstrate, on the same A100, an assembled
mass-operator throughput per unique DoF that exceeds the libCEED BP1 benchmark by up to a factor of four, the margin
growing with polynomial degree.

The entire solver is implemented in Triton~\cite{triton2019}, a tile-based GPU language originating in the
machine-learning community. This choice is deliberate: Triton's programming model operates on highly efficient dense,
power-of-two tiles, and takes care of coalescing, shared-memory bank conflicts, and the mapping of tile-level tensor
contractions onto tensor cores, addressing concerns that would otherwise demand hand-tuned CUDA. Triton also sidesteps
the CUDA toolchain entirely: its compiler translates the Python-embedded kernel source, through a sequence of
intermediate representations (MLIR), down to the GPU's native instruction set: PTX, NVIDIA's low-level format, or its
AMD counterpart; so the same kernel source runs on both NVIDIA and AMD hardware without a line of vendor-specific code.
Section~\ref{sec:implementation} details this implementation, including the decomposition of non-power-of-two
polynomial degrees into power-of-two tiles and a small-tile tensor-core extension contributed upstream as part of this
work.

The primary cost of redundancy is memory footprint, and it is concentrated at the lowest order. The overhead over the
assembled vector is the asymptotic factor $\big((p+1)/p\big)^3$, which reaches $8\times$ at $p=1$ but then collapses
steeply with degree: to $3.4\times$ at $p=2$, $2.4\times$ at $p=3$, and toward $1.5\times$ at $p=7$
(Table~\ref{tab:memory_overhead}), as the element interior comes to dominate the shared interface. It does not,
however, translate into a runtime penalty. Because the format removes the uncoalesced gather-scatter and the
destination-read that burden classical assembly, the additional bytes are streamed at the device's coalesced peak
rather than fetched through indirection, and the net effect is favourable across the \emph{entire} degree range: the
assembled operator already outperforms the libCEED BP1 baseline at $p=1$ and by a widening margin thereafter
(Section~\ref{sec:results}). The redundancy thus buys a larger but fully coalesced footprint in exchange for
eliminating the indirect traffic that otherwise dominates the operator.

The remainder of the paper is organized as follows. Section~\ref{sec:cfem_dg} introduces the discrete setting,
contrasts classical matrix-free evaluation with persistent cell-wise storage, and quantifies the memory overhead
together with initial throughput measurements on the BP1 mass-operator benchmark. Section~\ref{sec:solvers} develops
the primal-dual framework and proves the equivalence theorem for the flexible conjugate gradient iteration.
Section~\ref{sec:dss} presents the layered construction of the DSS operator: the dimensionally-split structured cascade
and the unstructured interface kernels, with the shadow-cell treatment of $h$-adaptive non-conformity deferred to
Appendix~\ref{sec:dss_hanging}. Section~\ref{sec:implementation} describes the implementation in the Triton language,
including the mapping of sum-factorized kernels onto tensor cores and the handling of non-power-of-two polynomial
degrees. Section~\ref{sec:results} reports numerical results on an NVIDIA A100, and Section~\ref{sec:conclusions}
concludes.


\section{Continuous FEM with Cell-Wise Storage}
\label{sec:cfem_dg}

We consider a domain $\Omega \subset \mathbb{R}^d$ partitioned into a triangulation $\mathcal{T}_h$ of tensor-product
elements: quadrilaterals in 2D ($d=2$) and hexahedra in 3D ($d=3$). The mesh originates from a relatively coarse,
conformal base grid (no hanging nodes), which fixes the topology of the domain; the computational grid is obtained from
it by a hierarchical sequence of refinement cycles, either global (uniform) or local (adaptive). Local refinement
inherently produces hanging nodes where cells of differing refinement levels abut. The \textit{active cells} are the
elements of the hierarchy without children (the finest cells in any given region), and it is on this collection that
the continuous finite element space of polynomial degree $p$ is built. The discrete problem then reads
\begin{equation}
    A u = \bar{f},
\end{equation}
where the global operator $A \colon \mathbb{V}_{CG} \to \mathbb{V}_{CG}^*$ represents the discrete bilinear form acting between the continuous primal space $\mathbb{V}_{CG}$ and its dual $\mathbb{V}_{CG}^*$. Accordingly, $u \in \mathbb{V}_{CG}$ and $\bar{f} \in \mathbb{V}_{CG}^*$ are the global vectors of degrees of freedom (DoFs) and the right-hand side, respectively.

\subsection{Classical matrix-free evaluation and its memory-access bottleneck}
In high-performance continuous finite element frameworks, the action of $A$ is evaluated without assembling a sparse
matrix. To cleanly formalize this matrix-free workflow, we must delineate both topologies (global vs.\ cell-wise) and
mathematical spaces (primal vs.\ dual). We denote the continuous global spaces as $\mathbb{V}_{CG}$ (primal) and
$\mathbb{V}_{CG}^*$ (dual), and their redundant cell-wise counterparts as $\mathbb{V}_{\text{cell}}$ and
$\mathbb{V}_{\text{cell}}^*$, in which every element stores its own copy of the DoFs on shared interfaces.

Transitions between these topologies in the classical matrix-free evaluation are managed by two operators:
\begin{itemize}
    \item $\mathcal{G} \colon \mathbb{V}_{CG} \to \mathbb{V}_{\text{cell}}$ is the global gather operator acting on primal spaces.
    \item $\mathcal{G}^T \colon \mathbb{V}_{\text{cell}}^* \to \mathbb{V}_{CG}^*$ is the formal adjoint of $\mathcal{G}$. It acts as the scatter (assembly) operator mapping local contributions back to the global topology.
\end{itemize}

With these operators established, the typical matrix-free evaluation of $A$ is formulated as follows:
\begin{enumerate}
    \item \textbf{Gather}: Extract local DoF values from the global primal vector $u$: $u_K = \mathcal{G}_K u$.
    \item \textbf{Evaluate}: Compute the local unassembled dual contribution independently: $\tilde{f}_K = A_K u_K$.
    \item \textbf{Scatter}: Accumulate (sum) the local unassembled dual results back into the global, assembled residual vector $\bar{f}$: $\bar{f} = \sum_K \mathcal{G}_K^T \tilde{f}_K$.
\end{enumerate}

While the arithmetic of step 2 is well understood and, with sum factorization, cheap, the cost of this pipeline on
modern GPUs is dominated by steps 1 and 3. The gather and scatter are \emph{indirect} memory accesses: each thread
follows an element-to-global index map into a vector whose ordering is dictated by the global mesh numbering, not by
the access pattern of the kernel. Consecutive threads therefore touch non-consecutive addresses, and the hardware
splits each warp-wide load into many separate memory transactions.

This lack of \emph{coalescing} is an important bottleneck. The nominal device bandwidth is attainable only when the
threads of a warp access consecutive words that the memory controller can serve in a single transaction. \footnote{On
    the A100 (nominal 2\,039\,GB/s) our cell-wise kernels reach up to 1\,774.5\,GB/s in FP64. By way of reference, a simple
    FP64 vector scaling in JAX/XLA~\cite{jax2018github} (a software stack that can safely be assumed highly optimized)
    measures roughly 1\,680\,GB/s.} An uncoalesced gather may fetch an entire cache line per useful word, discarding the
bulk of the transferred data, and the scatter is worse still: concurrent accumulation into shared interface DoFs
additionally requires atomic operations or mesh coloring. In practice, classical matrix-free kernels for unstructured
meshes are thus limited by transaction and atomics throughput rather than by raw memory bandwidth: the headline
bandwidth figure of the device is structurally out of reach.

Even if that bandwidth could be reached, it (and not the arithmetic throughput) would still be the binding constraint.
The hardware landscape summarized in Table~\ref{tab:gpu_comparison_full} explains why this paper optimizes memory
traffic rather than arithmetic. For each device we list the minimum arithmetic intensity (the FLOP/byte ratio below
which a kernel is memory-bound) obtained as the quotient of peak FP64 throughput and memory bandwidth. Even on the
A100, the most bandwidth-friendly of the recent NVIDIA datacenter parts, a kernel must perform about five
double-precision operations per transferred byte (roughly 40 per double-precision value) before the compute units
matter at all; on H100-class hardware and on AMD's CDNA accelerators the threshold lies at 10--15 FLOP/byte and beyond.

The arithmetic intensity of a sum-factorized operator is not fixed: it grows with the polynomial degree, since the work
per DoF rises while the data moved per DoF does not. At low degrees the operators fall far short of the thresholds
above. On the A100 and on Cartesian geometries, the kernels of this work are firmly memory-bound up to degree $p = 2$,
so the quantity that decides their performance is not the FLOP count but how fast the required bytes can be moved,
i.e.\ the fraction of peak bandwidth actually realized.

This reframes the optimization target. The lever is not to \emph{minimize} bytes moved: our cell-wise format in fact
moves \emph{more} data than an assembled scheme, since it stores interface DoFs redundantly. What it does instead is
make every access coalesced and eliminate the read-accumulate and atomic traffic of assembly, raising the achievable
bandwidth such that this gain outweighs the added redundancy. From $p = 3$ onward this no longer holds and the kernels
leave the memory-bound regime: the pipeline carries additional operations except dense matrix-matrix multiplications,
so the arithmetic actually issued outstrips the useful-work count and compute, rather than bandwidth, becomes the
limiter.

Although this work was initially designed to target high-order discretizations, the kernels remain memory-bound across
much of the degree range; the storage strategy pays off already from degree one. Wherever bandwidth is the binding
constraint, performance is governed by how the field data are laid out and accessed.

\begin{table}[htbp]
    \centering
    \begin{tabular}{l c c c}
        \toprule
        GPU Model  & Bandwidth & Peak FP64 (Tensor) & Min.\ AI    \\
                   & (GB/s)    & (TFLOPS)           & (FLOP/byte) \\
        \midrule
        A100 (SXM) & 2,039     & 19.5               & 9.56        \\
        H100 (SXM) & 3,350     & 67.0               & 20.00       \\
        H200 (SXM) & 4,800     & 67.0               & 13.96       \\
        B200 (SXM) & 8,000     & 40.0               & 5.00        \\
        \midrule
        MI250X     & 3,277     & 95.7               & 29.20       \\
        MI300X     & 5,300     & 163.4              & 30.83       \\
        MI325X     & 6,000     & 163.4              & 27.23       \\
        \bottomrule
    \end{tabular}
    \caption{NVIDIA and AMD GPU bandwidth, peak FP64 tensor/matrix throughput, and minimum arithmetic intensity.}
    \label{tab:gpu_comparison_full}
\end{table}

\subsection{Persistent cell-wise storage}
In classical implementations, the cell-wise representation $u_K$ is never formed as a persistent global object:
fragments of it are generated on-the-fly as threads read from the global vector, processed within cache, registers, or
shared memory, and scattered back into the global dual vector. In this work we adopt the opposite strategy: the field
data are kept \emph{persistently} in the cell-wise storage format $\mathbb{V}_{\text{cell}}$, and all mathematical
operations (operator evaluations, duality pairings, vector updates, and preconditioning) are performed directly on this
redundant representation. A global, non-redundant vector is never formed or traversed.

As a result, the storage layout is decoupled from the global mesh topology. Every element owns a fixed-size,
self-contained set of DoFs, so these sets can be arranged in memory in whatever order suits the hardware, and every
kernel addresses them through statically known offsets instead of an index map. Indirect addressing disappears from the
entire solver pipeline, and with it the transaction fragmentation described above: the layout can be chosen such that
warp-wide accesses are coalesced, which makes the nominal device bandwidth attainable for finite element kernels.

\begin{table}
    \centering
    \begin{tabular}{l c c c c c}
        \toprule
        Storage unit            & DoFs per edge & Stored DoFs & Unique DoFs & \makecell{Storage        \\ratio} & \makecell{Transfer\\ratio} \\
        \midrule
        single $\mathbb{Q}_3$   & 4             & 64          & 27          & 2.37              & 1.58 \\
        single $\mathbb{Q}_7$   & 8             & 512         & 343         & 1.49              & 0.99 \\
        \midrule
        $2^d$ of $\mathbb{Q}_1$ & 3             & 27          & 8           & 3.38              & 2.25 \\
        $2^d$ of $\mathbb{Q}_2$ & 5             & 125         & 64          & 1.95              & 1.30 \\
        $2^d$ of $\mathbb{Q}_3$ & 7             & 343         & 216         & 1.59              & 1.06 \\
        \bottomrule
    \end{tabular}
    \caption{Memory overhead of cell-wise storage in 3D, for a single hexahedral element (top) and for an optional
        $2^d$-element cluster packed into one storage unit (bottom). \emph{DoFs per edge} is the number of 1D nodes along one
        edge of the storage unit; \emph{Stored DoFs} is its cell-wise DoF count, and \emph{Unique DoFs} the asymptotic number of
        distinct DoFs in the same volume of an assembled vector. The storage ratio is their quotient, i.e.\ the redundancy
        factor of the format. The transfer ratio additionally accounts for the data moved during operator evaluation: since
        cell-wise evaluation writes its output directly instead of read-accumulating into an assembled vector, the format
        becomes bandwidth-\emph{cheaper} than a traditional assembled vector for $\mathbb{Q}_7$ (ratio below one).}
    \label{tab:memory_overhead}
\end{table}

We remark that the framework also permits packing a small structured cluster of elements (e.g.\ $2^d$ neighbouring
cells) into a single logical macro-element whose interior interface DoFs are shared rather than duplicated. This
reduces the storage redundancy (cf.\ Table~\ref{tab:memory_overhead}) at the cost of slightly larger local kernels;
none of the algorithms presented below depend on this packing, and we treat it as an optional compression of the same
storage scheme. We do not explore it further in this work, but it is a straightforward extension of the same ideas and
can be used to reduce the memory footprint of the cell-wise representation.

\subsection{Initial performance measurements}
\label{sec:volumetric_perf}

To give an early indication of what the storage paradigm delivers, we measure the action of two operators on the
cell-wise storage: the mass operator and the Laplace (stiffness) operator. Both act on the local DoF vector $u_K \in
    \mathbb{V}_{\text{cell}}$, the cell-wise block of element $K$ read directly from the persistent storage, and both are
evaluated in sum-factorized form on each tensor-product cell,
\begin{align}
    M_K u_K & = (V^T \otimes V^T \otimes V^T)\, h \,(V \otimes V \otimes V)\, u_K,
    \label{eq:mass_operator}                                                       \\
    L_K u_K & = (D_x M_y M_z + M_x D_y M_z + M_x M_y D_z)\, u_K,
    \label{eq:laplace_operator}
\end{align}
where $V$ is the 1D matrix evaluating nodal values at the quadrature points, $D$ and $M$ are the 1D derivative and mass
contractions, and the scalar $h$ collects the cell-size scaling, mimicking a Cartesian grid and avoiding
per-quadrature-point geometry data. The mass operator requires six 1D contractions, the Laplace operator seven, so the
latter carries roughly three times the arithmetic; for $p = 1, 2$ both use the even--odd decomposition and for
$p \ge 3$ they do not.

These two operators bracket the behaviour of the cell-wise format. On a naive arithmetic-intensity count one would
expect the mass operator to be the worst case, since its arithmetic load is minimal and a firmly memory-bound kernel
pays in full for the redundant data, as every duplicated DoF is extra traffic. However, the opposite holds: the mass
contractions map poorly onto the tensor-core (MMA) tiles, so the kernel issues substantial extra arithmetic and is not,
in the end, limited by memory traffic. The Laplace operator, despite its higher arithmetic intensity, tiles efficiently
onto the tensor cores and sustains a high floating-point rate, so it is the operator that is memory-bound, rendering it
the worst case for the cell-wise format and hence the more informative measurement.

Tables~\ref{tab:perf_teaser} and~\ref{tab:perf_teaser_laplace} report the maximum achieved throughput for the full
range of polynomial degrees $p = 1, \dots, 7$ supported by our implementation. We quote the rate both in terms of
unique DoFs (the figure comparable across implementations) and cell-wise, duplicated DoFs (the figure reflecting the
work actually performed), together with the attained memory throughput. The operation count per element is nearly
identical to a classical matrix-free implementation, but the cell-wise variant writes its output directly instead of
read-accumulating into an assembled vector. For the Laplace operator we additionally report two arithmetic rates. The
\emph{theory} TFLOP/s is the useful work: it multiplies the cell-wise DoF rate by the floating-point operations the
sum-factorized operator of~\eqref{eq:laplace_operator} must perform per DoF. The \emph{achieved} TFLOP/s instead counts
the operations the hardware actually issues; on the tensor-core path used for $p \ge 3$ the contraction tiles are
padded, so the hardware performs more arithmetic than is necessary and the achieved rate exceeds the theoretical one.
For $p = 1, 2$ no tensor cores and hence no padding are involved, so the two rates coincide.

The tables bear this out: at the upper degrees the Laplace operator delivers the \emph{higher} throughput per unique
DoF of the two (cf.\ Tables~\ref{tab:perf_teaser} and~\ref{tab:perf_teaser_laplace}), its three-fold arithmetic
absorbed almost for free by the better-filled tensor-core tiles. The reference row of
Table~\ref{tab:perf_teaser_laplace} lists the achieved Laplace rates of Cui et al.~\cite{cui2025implementation}
(reported for $p = 2, \dots, 8$ and aligned here to the matching degree); our useful-work (\emph{theory}) rate exceeds
theirs at every common degree. At $p = 2$ this margin is moreover understated: the reference figures do not employ the
even--odd factorization, so they count arithmetic our formulation avoids, and our advantage in actual throughput
(GDoF/s) is correspondingly larger than the TFLOP/s rows alone convey. A full benchmark campaign, including layout
comparisons and the direct stiffness summation (DSS) kernels, follows in Section~\ref{sec:results}.

\begin{table}
    \centering
    \begin{tabular}{l c c c c c c c}
        \toprule
                                     & $\mathbb{Q}_1$ & $\mathbb{Q}_2$ & $\mathbb{Q}_3$ & $\mathbb{Q}_4$ & $\mathbb{Q}_5$ & $\mathbb{Q}_6$ & $\mathbb{Q}_7$ \\
        \midrule
        GDoF/s (unique)              & 14.28          & 32.98          & 39.76          & 36.88          & 30.56          & 30.28          & 40.43          \\
        GDoF/s (cell-wise)           & 111.68         & 110.16         & 93.43          & 71.55          & 52.57          & 47.85          & 60.24          \\
        Ratio (cell-wise/unique)     & 7.82           & 3.34           & 2.35           & 1.94           & 1.72           & 1.58           & 1.49           \\
        \hline
        Memory throughput {[}GB/s{]} & 1774           & 1756           & 1526           & 1187           & 841            & 766            & 964            \\
        \bottomrule
    \end{tabular}
    \caption{Maximum measured throughput of the mass operator (BP1-like benchmark) in cell-wise storage. Rates are
        reported per unique DoF (comparable to assembled implementations) and per cell-wise, duplicated DoF; the last row
        gives the attained memory throughput. The ratio of cell-wise GDoF/s to memory throughput is essentially constant
        at $1/16$ across all degrees, indicating that exactly two doubles (one read, one write) are transferred per DoF.}
    \label{tab:perf_teaser}
\end{table}

\begin{table}
    \centering
    \begin{tabular}{l c c c c c c c}
        \toprule
                                                       & $\mathbb{Q}_1$ & $\mathbb{Q}_2$ & $\mathbb{Q}_3$ & $\mathbb{Q}_4$ & $\mathbb{Q}_5$ & $\mathbb{Q}_6$ & $\mathbb{Q}_7$ \\
        \midrule
        GDoF/s (cell-wise)                             & 110.2          & 108.3          & 99.69          & 73.48          & 65.22          & 62.41          & 88.82          \\
        GDoF/s (unique)                                & 14.09          & 32.43          & 42.42          & 37.88          & 37.92          & 39.50          & 59.61          \\
        Memory throughput {[}GB/s{]}                   & 1 751          & 1727           & 1628           & 1219           & 1044           & 999            & 1421           \\
        \hline
        achieved  TFLOP/s                              & 2.81           & 3.15           & 8.14           & 7.91           & 7.08           & 10.62          & 10.15          \\
        theory TFLOP/s                                 & 2.81           & 3.15           & 5.58           & 5.14           & 5.48           & 6.12           & 9.95           \\
        reference TFLOP/s~\cite{cui2025implementation} & ---            & 1.30           & 1.32           & 2.02           & 2.21           & 1.90           & 1.90           \\
        \bottomrule
    \end{tabular}
    \caption{Maximum measured throughput of the Laplace operator in cell-wise storage, evaluated as
        $L u = (D_x M_y M_z + M_x D_y M_z + M_x M_y D_z)\,u$. Rates are reported per unique DoF (comparable to assembled
        implementations) and per cell-wise, duplicated DoF; the last row gives the attained memory throughput. The ratio
        of cell-wise GDoF/s to memory throughput is essentially constant at $1/16$ across all degrees, indicating that
        exactly two doubles (one read, one write) are transferred per DoF. The \emph{theory} row gives the useful
        arithmetic throughput obtained from the cell-wise rate and the operation count of the sum-factorized operator,
        whereas the \emph{achieved} row counts the operations the hardware actually issues, including padding. The
        \emph{reference} row reproduces the achieved Laplace rates of Cui et al.~\cite{cui2025implementation}, reported
        for $p = 2, \dots, 8$ and aligned here to the matching degree (their $p=8$ value falls outside our range);
        unlike ours, their kernels do not use the even--odd factorization at low order.}
    \label{tab:perf_teaser_laplace}
\end{table}

\section{Krylov Subspace Methods with explicit Primal-Dual pairing}
\label{sec:solvers}

In a cell-wise (DG-like) storage scheme for continuous finite elements, the global degrees of freedom are duplicated
across shared element interfaces and hanging nodes. This duplication might appear to be a liability: every operation
that touches shared DoFs seems to require synchronization of the redundant copies. However, the opposite holds,
provided one maintains a mathematical distinction between two kinds of objects living in the redundant storage:
continuous \emph{primal} fields, whose duplicated entries agree by construction, and broken \emph{dual} vectors of
local cell-wise integrals, which are meaningful without any summation. Krylov subspace methods, when written so that
every inner product pairs one object of each kind, never need to convert between the two representations. We exploit
this to reformulate the (flexible) conjugate gradient method so that the operator application, all vector updates, and
all duality pairings are evaluated without any inter-element communication, and we prove that the resulting iteration
is identical, iterate by iterate, to the classical solver applied to the assembled system. The only component that must
bridge the two representations is the preconditioner, which we address in Section~\ref{sec:trivial_preconditioner}.

\subsection{Constraint Operator and Duality Pairing Identity}
Let $\mathbb{V}_{\text{cell}}$ denote the high-dimensional broken space encompassing all local degrees of freedom
stored independently on each cell. We define the continuous subspace as the image of the gather operator
$\mathcal{G}(\mathbb{V}_{CG}) \subset \mathbb{V}_{\text{cell}}$, representing continuous fields.

We introduce a constraint matrix $C \colon \mathbb{V}_{\text{cell}} \to \mathbb{V}_{\text{cell}}$ that projects any
broken vector onto the continuous subspace. The exact formulation of $C$ can vary (e.g., averaging values on shared
faces), but its defining mathematical property is that it is idempotent: $C = C^2$, with its range being exactly
$\mathcal{G}(\mathbb{V}_{CG})$. Consequently, any vector $u \in \mathbb{V}_{\text{cell}}$ is a valid continuous primal
vector if and only if it is a fixed point of the constraint operator:
\begin{equation} \label{eq:constraint}
    u = C u.
\end{equation}
Because $C$ can represent any idempotent projection, Eq.~\eqref{eq:constraint} handles complex spatial constraints, ensuring that hanging nodes are slaved to their coarse neighbors.

Conversely, the dual space $\mathbb{V}_{\text{cell}}^*$ is the space of linear functionals (namely, the raw,
unassembled local residuals and fluxes. Dual vectors, denoted $\tilde{f} \in \mathbb{V}_{\text{cell}}^*$, represent
local integrals and do not require inter-element summation to be valid. The true assembled global residual
representation is recovered via the transpose of the constraint matrix, $C^T \tilde{f}$, which physically represents
the accumulation across interfaces.

The advantage of cell-wise storage is evident when evaluating the standard Euclidean duality pairing $\langle \cdot,
    \cdot \rangle$ between an unassembled dual vector and a continuous primal vector.

\begin{lemma}[Duality Pairing Identity]
    \label{lem:inner_product}
    Let $\tilde{f} \in \mathbb{V}_{\text{cell}}^*$ be an unassembled dual vector and $u \in \mathbb{V}_{\text{cell}}$ be a continuous primal vector satisfying $u = C u$. The duality pairing between $\tilde{f}$ and $u$ satisfies:
    \begin{equation} \label{eq:inner_product}
        \langle \tilde{f}, u \rangle = \langle C^T \tilde{f}, u \rangle.
    \end{equation}
\end{lemma}

\begin{proof}
    Substituting the continuity constraint $u = C u$ into the duality pairing and applying the property of the matrix transpose yields:
    \begin{equation*}
        \langle \tilde{f}, u \rangle = \langle \tilde{f}, C u \rangle = \langle C^T \tilde{f}, u \rangle.
    \end{equation*}
\end{proof}

Note that the fixed-point property of continuous fields can be written as $C \mathcal{G} = \mathcal{G}$, since
$\text{Im}(\mathcal{G}) = \text{Im}(C)$ and $C$ acts as the identity on its range. An immediate consequence of
Lemma~\ref{lem:inner_product} is that the pairing is also independent of the particular choice of $C$ and recovers the
pairing on the minimal global space: for $u = \mathcal{G} u^{\text{global}}$ with $u^{\text{global}} \in
    \mathbb{V}_{CG}$,
\begin{equation} \label{eq:pairing_global}
    \langle \tilde{f}, u \rangle = \langle \tilde{f}, \mathcal{G} u^{\text{global}} \rangle = \langle \mathcal{G}^T \tilde{f}, u^{\text{global}} \rangle,
\end{equation}
i.e.\ the local pairing of an unassembled dual vector with a continuous primal vector equals the global pairing of the assembled functional $\mathcal{G}^T \tilde{f}$ with the underlying global coefficient vector.

This simple identity allows us to compute the global duality pairing using only local data. Given the primal vector $u$
satisfying the constraints, pairing it with the raw cell-wise residuals is enough to recover the global result. This
allows us to calculate accurate global norms and duality pairings without ever needing to perform an explicit assembly.

With this duality established, we can formally state the governing equation of the continuous finite element problem.
The true continuous problem seeks a primal vector $u \in \mathbb{V}_{\text{cell}}$ satisfying the continuity constraint
$u = C u$ such that the globally assembled residual vanishes for any continuous test function. Let $\tilde{b} \in
    \mathbb{V}_{\text{cell}}^*$ be the unassembled right-hand side. The globally assembled linear system is formally
expressed as:
\begin{equation} \label{eq:constrained_system}
    C^T A C u = C^T \tilde{b},
\end{equation}
where $C^T A C \colon \mathbb{V}_{\text{cell}} \to \mathbb{V}_{\text{cell}}^*$ represents the complete system operator embedded within the redundant space. This constrained formulation is equivalent to the classical minimal global system $A_{CG} u^{\text{global}} = b_{CG}$ posed over the continuous space $\mathbb{V}_{CG}$, where $A_{CG} = \mathcal{G}^T A \mathcal{G}$ is the explicitly assembled operator, $b_{CG} = \mathcal{G}^T \tilde{b}$, and the corresponding solutions map identically via $u = \mathcal{G} u^{\text{global}}$.

To deploy this identity within a preconditioned Krylov solver, we must distinguish how operators map between the primal
and dual spaces. Central to the problem definition is the local operator $A$, which maps a continuous primal vector to
an unassembled dual residual. In a preconditioned solver, an additional ingredient is the preconditioner $P$, which
attempts to invert $A$ approximately by mapping a dual residual back to a primal correction. Concretely:
\begin{itemize}
    \item \textbf{The Local Operator ($A$):} The application of the local element matrices maps a continuous primal search direction $p$ to an unassembled dual flux $\tilde{q}$. No inter-element communication is performed. $A \colon \mathbb{V}_{\text{cell}} \to \mathbb{V}_{\text{cell}}^*$.
    \item \textbf{The Preconditioner ($P$):} The preconditioner is responsible for bridging the two spaces. It represents an operator that, given a dual vector $\tilde{r} \in \mathbb{V}_{\text{cell}}^*$ (in cell-wise format), returns a primal vector $z \in \mathbb{V}_{\text{cell}}$ that is an approximation of the solution satisfying the constraints defined by $C$. $P \colon \mathbb{V}_{\text{cell}}^* \to \mathbb{V}_{\text{cell}}$.
\end{itemize}

\subsection{The Conjugate Gradient Algorithm}
The Conjugate Gradient method naturally arises from the minimization of the quadratic energy functional $\mathcal{J}(u)
    = \frac{1}{2} \langle A u, u \rangle - \langle \tilde{b}, u \rangle$ over the continuous primal space. The variation
(gradient) of this functional yields the residual $\tilde{r} = \tilde{b} - A u$. Because the operator $A$ maps
continuous fields to distributions, this residual resides in the dual space $\mathbb{V}_{\text{cell}}^*$.

This poses a mathematical challenge: we cannot directly update the primal solution $u$ using the dual residual
$\tilde{r}$, as adding a functional to a field is an ill-defined operation. The preconditioner $P \colon
    \mathbb{V}_{\text{cell}}^* \to \mathbb{V}_{\text{cell}}$ serves as the mathematical bridge between these spaces. It
maps the unassembled dual residual back to a continuous primal correction $z = P(\tilde{r})$ while implicitly defining
the preconditioned inner product structure which governs the Krylov subspace.

To minimize the energy functional along a continuous search direction $p_k \in \mathbb{V}_{\text{cell}}$, the step size
$\alpha_k$ is obtained by enforcing orthogonality between the updated dual residual and the primal search direction
($\langle \tilde{r}_{k+1}, p_k \rangle = 0$). In standard PCG, this yields:
\begin{equation} \label{eq:alpha}
    \alpha_k = \frac{\langle \tilde{r}_k, z_k \rangle}{\langle \tilde{q}_k, p_k \rangle},
\end{equation}
where $\tilde{q}_k = A p_k$. In the context of cell-wise storage, the numerator pairs a dual residual with a primal correction, and the denominator pairs a dual flux with a primal search direction. By Lemma~\ref{lem:inner_product}, both of these duality pairings can be evaluated using only local, unassembled cell-wise operations.

In principle, the primal-dual structure above applies verbatim to standard PCG. Our motivation for going beyond it is
hardware-driven: GPUs offer substantially higher throughput in reduced precision (single or even half), and the natural
place to exploit this is inside the preconditioner~\cite{goddeke2007performance}, e.g.\ by running the multigrid
V-cycle in lower precision or varying its cycle structure between iterations. Such a preconditioner is no longer a
fixed linear operator. Standard PCG, however, relies on a static preconditioner to maintain the mutual $A$-conjugacy of
the search directions; if $P$ changes between iterations, its convergence degrades or fails outright. We therefore
employ the Flexible Conjugate Gradient (FCG) method proposed by Notay~\cite{notay2000flexible}. Since the two methods
differ only in the formula for $\beta_k$, we state them in primal-dual form as a single Algorithm~\ref{alg:fcg}, with
the two $\beta_k$ variants given side by side: in both cases every scalar is a duality pairing of one dual and one
primal vector, computable from local data by Lemma~\ref{lem:inner_product}. Since the primal-dual reasoning is
identical for PCG and FCG, we present the proofs for the more general FCG variant only.

FCG modifies the search direction update to explicitly enforce $A$-conjugacy against the varying preconditioner,
removing the reliance on a static inner product geometry. The updated search direction is $p_{k+1} = z_{k+1} + \beta_k
    p_k$, where:
\begin{equation} \label{eq:beta}
    \beta_k = \frac{- \langle \tilde{q}_k, z_{k+1} \rangle}{\langle \tilde{q}_k, p_k \rangle}.
\end{equation}
FCG preserves the necessary primal-dual alignment. The new preconditioned residual $z_{k+1}$ is primal, and the local operator output $\tilde{q}_k$ is dual. Thus, the FCG numerator $\langle \tilde{q}_k, z_{k+1} \rangle$ computes the global $A$-conjugacy locally, circumventing the need to explicitly assemble the operator output.

This flexibility comes at the modest price of one additional scalar product per iteration. In standard PCG the
numerator of $\beta_k$, $\langle \tilde{r}_{k+1}, z_{k+1} \rangle$, is reused as the numerator of $\alpha_{k+1}$ in the
next iteration, so $\beta_k$ requires no new reduction. The FCG numerator $\langle \tilde{q}_k, z_{k+1} \rangle$, by
contrast, must be computed in addition to it, while the shared denominator $\langle \tilde{q}_k, p_k \rangle$ is
already available from the computation of $\alpha_k$. Since each iteration is dominated by the operator and
preconditioner applications, this extra local pairing is negligible in practice.

By adopting this primal-dual algorithmic structure, the entire Krylov solver is reduced to parallel, local cell-wise
operations. The requirement for communication, interface accumulation, and hanging-node resolution is completely
abstracted away from the Krylov subspace projection and delegated exclusively to the preconditioner. The resulting
routine is summarized in Algorithm~\ref{alg:fcg}, covering both the static (PCG) and flexible (FCG) choices of
$\beta_k$. We establish that this local formulation is equivalent to solving the globally assembled system, by the
theorem below.

\begin{theorem}[Equivalence of Communication-Free FCG]\label{thm:equivalence}
    Assume the local operator $A \colon \mathbb{V}_{\text{cell}} \to \mathbb{V}_{\text{cell}}^*$ is symmetric. Let the initial guess $u_0 \in \mathbb{V}_{\text{cell}}$ satisfy $u_0 = \mathcal{G} u_0^{\text{global}}$ for some $u_0^{\text{global}} \in \mathbb{V}_{CG}$, and assume the preconditioner is \emph{consistent}: at every iteration $k$ there exists a global preconditioner $P_{CG,k} \colon \mathbb{V}_{CG}^* \to \mathbb{V}_{CG}$ such that
    \begin{equation} \label{eq:precond_consistency}
        P_k = \mathcal{G} \, P_{CG,k} \, \mathcal{G}^T.
    \end{equation}
    In particular, $P_k$ acts on a dual vector only through its assembly $\mathcal{G}^T \tilde{r}$, and its output is continuous, $\text{Im}(P_k) \subseteq \text{Im}(C)$. Then the sequence of primal iterates $\{u_k\}$ generated by the communication-free FCG algorithm (Algorithm~\ref{alg:fcg}) is identical to the sequence of iterates $\{u_k^{\text{global}}\} \subset \mathbb{V}_{CG}$ produced by standard FCG applied to the explicitly assembled minimal global system $A_{CG} u^{\text{global}} = b_{CG}$ with the preconditioners $P_{CG,k}$: $u_k = \mathcal{G} u_k^{\text{global}}$ for all $k \geq 0$.
\end{theorem}

\begin{proof}
    Throughout, denote by $u_k^{\text{global}}, p_k^{\text{global}}, z_k^{\text{global}} \in \mathbb{V}_{CG}$ and $r_k^{\text{global}} \in \mathbb{V}_{CG}^*$ the quantities produced by standard FCG applied to $A_{CG} u^{\text{global}} = b_{CG}$ with the preconditioners $P_{CG,k}$, and recall $A_{CG} = \mathcal{G}^T A \mathcal{G}$, $b_{CG} = \mathcal{G}^T \tilde{b}$. We prove by induction the following correspondence between the two iterations:
    \begin{equation} \label{eq:induction_hypothesis}
        u_k = \mathcal{G} u_k^{\text{global}}, \qquad
        p_k = \mathcal{G} p_k^{\text{global}}, \qquad
        \mathcal{G}^T \tilde{r}_k = r_k^{\text{global}}.
    \end{equation}

    \emph{Base case.} By assumption $u_0 = \mathcal{G} u_0^{\text{global}}$. For the residual,
    \begin{equation*}
        \mathcal{G}^T \tilde{r}_0
        = \mathcal{G}^T (\tilde{b} - A u_0)
        = \mathcal{G}^T \tilde{b} - \mathcal{G}^T A \mathcal{G} u_0^{\text{global}}
        = b_{CG} - A_{CG} u_0^{\text{global}}
        = r_0^{\text{global}}.
    \end{equation*}
    By consistency~\eqref{eq:precond_consistency}, $z_0 = P_0 \tilde{r}_0 = \mathcal{G} P_{CG,0} \mathcal{G}^T \tilde{r}_0 = \mathcal{G} P_{CG,0} r_0^{\text{global}} = \mathcal{G} z_0^{\text{global}}$, and hence $p_0 = z_0 = \mathcal{G} z_0^{\text{global}} = \mathcal{G} p_0^{\text{global}}$.

    \emph{Equality of the scalars.} Assume~\eqref{eq:induction_hypothesis} and $z_k = \mathcal{G} z_k^{\text{global}}$ hold at step $k$. All vectors $u_k, p_k, z_k$ lie in $\text{Im}(\mathcal{G}) = \text{Im}(C)$ and are therefore fixed points of $C$, so Lemma~\ref{lem:inner_product} in the form~\eqref{eq:pairing_global} applies to every duality pairing in the algorithm. With $\tilde{q}_k = A p_k$:
    \begin{align*}
        \langle \tilde{r}_k, z_k \rangle     & = \langle \mathcal{G}^T \tilde{r}_k, z_k^{\text{global}} \rangle = \langle r_k^{\text{global}}, z_k^{\text{global}} \rangle,                                        \\
        \langle \tilde{q}_k, p_k \rangle     & = \langle \mathcal{G}^T A \mathcal{G} p_k^{\text{global}}, p_k^{\text{global}} \rangle = \langle A_{CG}\, p_k^{\text{global}}, p_k^{\text{global}} \rangle,         \\
        \langle \tilde{q}_k, z_{k+1} \rangle & = \langle \mathcal{G}^T A \mathcal{G} p_k^{\text{global}}, z_{k+1}^{\text{global}} \rangle = \langle A_{CG}\, p_k^{\text{global}}, z_{k+1}^{\text{global}} \rangle.
    \end{align*}
    These are exactly the numerators and denominators of $\alpha_k$ and $\beta_k$ in the global algorithm, hence $\alpha_k = \alpha_k^{\text{global}}$ and $\beta_k = \beta_k^{\text{global}}$. (Symmetry of $A$ guarantees, as in the standard setting, that $A_{CG}$ is symmetric and the FCG scalars are well defined.)

    \emph{Induction step.} Using $\alpha_k = \alpha_k^{\text{global}}$ and linearity of $\mathcal{G}$ and $\mathcal{G}^T$:
    \begin{align*}
        u_{k+1}                       & = u_k + \alpha_k p_k = \mathcal{G}\,(u_k^{\text{global}} + \alpha_k p_k^{\text{global}}) = \mathcal{G} u_{k+1}^{\text{global}},                                                 \\
        \mathcal{G}^T \tilde{r}_{k+1} & = \mathcal{G}^T \tilde{r}_k - \alpha_k \mathcal{G}^T A \mathcal{G} p_k^{\text{global}} = r_k^{\text{global}} - \alpha_k A_{CG}\, p_k^{\text{global}} = r_{k+1}^{\text{global}}.
    \end{align*}
    Consistency~\eqref{eq:precond_consistency} then gives $z_{k+1} = \mathcal{G} P_{CG,k+1} \mathcal{G}^T \tilde{r}_{k+1} = \mathcal{G} z_{k+1}^{\text{global}}$, and with $\beta_k = \beta_k^{\text{global}}$,
    \begin{equation*}
        p_{k+1} = z_{k+1} + \beta_k p_k = \mathcal{G}\,(z_{k+1}^{\text{global}} + \beta_k p_k^{\text{global}}) = \mathcal{G} p_{k+1}^{\text{global}},
    \end{equation*}
    which establishes~\eqref{eq:induction_hypothesis} at step $k+1$ and completes the induction.
\end{proof}

\begin{remark}
    Continuity of the preconditioner output alone ($\text{Im}(P) \subseteq \text{Im}(C)$) is \emph{not} sufficient for the theorem: it keeps the iterates in the continuous subspace, but the local preconditioner could still distinguish between two unassembled residuals with the same assembly $\mathcal{G}^T \tilde{r}$, in which case its iterates correspond to no assembled solver. The factorization~\eqref{eq:precond_consistency} excludes this. All preconditioners constructed in this work satisfy it by design: the trivial preconditioner $S = \mathcal{G} \mathcal{I}_{\text{Riesz}} \mathcal{G}^T$ of Section~\ref{sec:trivial_preconditioner} matches~\eqref{eq:precond_consistency} with $P_{CG} = \mathcal{I}_{\text{Riesz}}$, the Jacobi variant with $P_{CG} = D_{CG}^{-1} \mathcal{I}_{\text{Riesz}}$, and the multigrid V-cycle consumes its input only through the DSS operator $S$, i.e.\ through $\mathcal{G}^T \tilde{r}$.
\end{remark}

\begin{remark}[Beyond DSS]\label{rem:beyond_dss}
    The theorem does not single out DSS as the mechanism enforcing continuity: any preconditioner of the
    form~\eqref{eq:precond_consistency} qualifies, i.e.\ any operator that consumes the residual only through its
    assembly $\mathcal{G}^T \tilde{r}$ and returns a vector in the continuous subspace (or, more generally, in the
    constrained subspace $\text{Im}(\mathcal{G})$ when hanging-node or boundary constraints are present). In this work
    continuity is restored by explicit DSS sweeps inside the smoother, but overlapping patch smoothers of Schwarz
    type~\cite{pavarino1993additive,ArnoldFalkWinther00,fischer2000overlapping,WitteArndtKanschat21,brubeck2021scalable,wichrowski2025local}
    achieve the same effect structurally: each patch solve acts on assembled patch-local data and its correction is
    continuous by construction, so the factorization~\eqref{eq:precond_consistency} holds with $P_{CG}$ the assembled
    patch smoother. The communication-free Krylov framework is therefore compatible with this entire class of
    smoothers without modification.
\end{remark}

As a direct corollary from the theorem above, provided the standard explicitly assembled FCG algorithm converges for
the given problem, the communication-free algorithm minimizes the global energy functional and converges to the true
continuous solution $u = \mathcal{G} u^{\text{global}}$.

\begin{algorithm}
    \caption{Communication-Free (Flexible) Preconditioned Conjugate Gradients}
    \label{alg:fcg}
    \begin{algorithmic}[1]
        \Require Initial guess $u_0$ (Primal), RHS $\tilde{b}$ (Dual), Local Operator $A$, Preconditioner $P$
        \State $\tilde{q}_{\text{tmp}} \gets A u_0$ \Comment{Local $A$: Primal $\to$ Dual}
        \State $\tilde{r}_0 \gets \tilde{b} - \tilde{q}_{\text{tmp}}$ \Comment{Local Update: Dual}
        \State $z_0 \gets P(\tilde{r}_0)$ \Comment{Preconditioner: Dual $\to$ Primal}
        \State $p_0 \gets z_0$ \Comment{Local Copy: Primal}
        \State $k \gets 0$
        \While{not converged}
        \State $\tilde{q}_k \gets A p_k$ \Comment{Local $A$: Primal $\to$ Dual}
        \State $\alpha_k \gets \frac{\langle \tilde{r}_k, z_k \rangle}{\langle \tilde{q}_k, p_k \rangle}$ \Comment{Local Duality Pairings: (Dual $\cdot$ Primal)}
        \State $u_{k+1} \gets u_k + \alpha_k p_k$ \Comment{Local Update: Primal}
        \State $\tilde{r}_{k+1} \gets \tilde{r}_k - \alpha_k \tilde{q}_k$ \Comment{Local Update: Dual}
        \If{convergence criteria met}
        \State \textbf{break}
        \EndIf
        \State $z_{k+1} \gets P(\tilde{r}_{k+1})$ \Comment{Preconditioner: Dual $\to$ Primal}
        \State $\beta_k \gets
            \begin{cases}
                \dfrac{\langle \tilde{r}_{k+1}, z_{k+1} \rangle}{\langle \tilde{r}_k, z_k \rangle} & \text{(PCG, static $P$)}   \\[1.2em]
                \dfrac{- \langle \tilde{q}_k, z_{k+1} \rangle}{\langle \tilde{q}_k, p_k \rangle}   & \text{(FCG, flexible $P$)}
            \end{cases}$ \Comment{Local Duality Pairings: (Dual $\cdot$ Primal)}
        \State $p_{k+1} \gets z_{k+1} + \beta_k p_k$ \Comment{Local Update: Primal}
        \State $k \gets k + 1$
        \EndWhile
    \end{algorithmic}
\end{algorithm}

\subsection{Simple Preconditioners and Direct Stiffness Summation (DSS)}
\label{sec:trivial_preconditioner}

As established in Algorithm~\ref{alg:fcg}, the communication-free FCG solver inherently requires a preconditioner $P
    \colon \mathbb{V}_{\text{cell}}^* \to \mathbb{V}_{\text{cell}}$ to map the unassembled dual residual back to a
continuous primal correction. A natural question arises: what constitutes the simplest possible, or \emph{trivial},
preconditioner in this framework?

Algebraically, one might consider the identity matrix $P = I$. However, directly equating a functional to a field
without a change of basis is mathematically ill-posed. The \emph{identity} operation between the dual space and the
primal space is the Riesz isomorphism. Because the unassembled dual residuals in $\mathbb{V}_{\text{cell}}^*$ represent
partial, localized integrals, applying this mapping requires a topological transition.

To formalize this, we conceptually project the operations through the continuous global topology. First, applying the
scatter operator $\mathcal{G}^T \colon \mathbb{V}_{\text{cell}}^* \to \mathbb{V}_{CG}^*$ performs an additive assembly
across shared element boundaries, reconstructing a complete, globally defined functional. Within this assembled space,
the discrete Euclidean Riesz isomorphism $\mathcal{I}_{\text{Riesz}} \colon \mathbb{V}_{CG}^* \to \mathbb{V}_{CG}$ acts
algebraically as a numerical identity on the coefficient vectors, reinterpreting the accumulated integrals as
continuous nodal values. Finally, the primal gather operator $\mathcal{G} \colon \mathbb{V}_{CG} \to
    \mathbb{V}_{\text{cell}}$ extracts this assembled $C^0$-continuous field back into the redundant block format.

This composite mapping defines the operator
\begin{equation} \label{eq:dss_operator}
    S = \mathcal{G} \mathcal{I}_{\text{Riesz}} \mathcal{G}^T \colon \mathbb{V}_{\text{cell}}^* \to \mathbb{V}_{\text{cell}},
\end{equation}
which is known as \emph{Direct Stiffness Summation} (DSS). The trivial preconditioner, which plays the role of the identity in the assembled setting, is given by $P = S$. Note that $S$ satisfies the continuity requirement of the equivalence theorem by construction: its image is exactly $\mathcal{G}(\mathbb{V}_{CG}) = \text{Im}(C)$.

This composite formulation also extends to basic relaxation schemes, such as the Jacobi preconditioner, which requires
an inverse diagonal modifier $D_{CG}^{-1}$. In a classical assembled framework, scaling must occur in the global space
before gathering. However, the diagonal scaling intertwines with the gather operator: let $d \in \mathbb{V}_{CG}$ be
the vector of diagonal entries of $D_{CG}$, and let $D_{\text{cell}} = \operatorname{diag}(\mathcal{G} d)$ be the
diagonal operator on $\mathbb{V}_{\text{cell}}$ whose entries are the gathered copies of $d$. Since $\mathcal{G}$
merely replicates entries, scaling before gathering equals gathering and then scaling each copy by the same factor,
i.e.\ $\mathcal{G} D_{CG}^{-1} = D_{\text{cell}}^{-1} \mathcal{G}$. Consequently,
\begin{equation}
    \mathcal{G} D_{CG}^{-1} \mathcal{I}_{\text{Riesz}} \mathcal{G}^T = D_{\text{cell}}^{-1} (\mathcal{G} \mathcal{I}_{\text{Riesz}} \mathcal{G}^T) = D_{\text{cell}}^{-1} S.
\end{equation}
This justifies evaluating the unscaled DSS first, and subsequently applying the pointwise scaling directly onto the localized cell vectors (i.e., $P = D_{\text{cell}}^{-1} S$).

The inter-cell communication is effectively moved into the preconditioner, while the outer Krylov solver remains local.
While the operator $S$ provides the essential bridge for Krylov preconditioners and the multigrid smoother, its
definition explicitly relies on $\mathcal{G}$ and $\mathcal{G}^T$. Executing these sparse mappings translates to severe
gather-scatter memory latency on modern hardware. The following section details how our algorithmic design allows for
circumventing the global memory bottleneck.

\section{Algorithmic Design of the DSS Operator}
\label{sec:dss}

The objective of our algorithmic design is to eliminate the global gather-scatter operations. Rather than routing data
through the indirect indices of a global vector, our DSS algorithms exploit the dense localization of the block layout:
partial sums are exchanged directly across cell interfaces via symmetric point-to-point transmissions, reproducing the
action of $S$ while operating within the high-bandwidth, vectorized $\mathbb{V}_{\text{cell}}$ storage format. The
construction is layered in three stages of increasing topological generality, each presented in its own subsection: a
dimensionally-split sweep that resolves all interfaces interior to a structured block
(Section~\ref{sec:dss_structured}) and a set of kernels for the unstructured interfaces between macro-blocks
(Section~\ref{sec:dss_unstructured}). The stages act on disjoint sets of degrees of freedom and compose into the full
operator $S$. A pre-/post-processing wrapper that further reduces non-conforming interfaces of $h$-adaptive meshes to
the conforming case is deferred to Appendix~\ref{sec:dss_hanging}.

\subsection{Base Algorithm for Structured Grids}
\label{sec:dss_structured}

To achieve the high throughput required for modern GPU architectures, our approach decomposes the structured block
interior summation into a sequence of dimensionally-split passes along the principal axes. In a 3D triangulation, a
single degree of freedom at a cell corner is shared by $2^3 = 8$ adjacent elements. While a traditional matrix-free
implementation necessitates a complex gather-scatter operation to aggregate all eight contributions simultaneously, our
multi-pass strategy achieves the same result through local \emph{one-to-one} exchanges: a single pass along axis
$\alpha$ visits every interior interface normal to $\alpha$, sums the two coincident copies of each interface DoF, and
writes the sum back to both sides, so that after the pass both copies hold the assembled value of that face pair.

The key property is that vertex and edge DoFs are assembled \emph{as a byproduct} of these face passes, without ever
being addressed explicitly. Consider, in 2D, a vertex shared by four elements (Figure~\ref{fig:dss_propagation}). Each
element initially holds its own partial contribution. The $y$-pass sums the copies across the two horizontal
interfaces, after which each copy holds the sum of its vertical pair of elements. The subsequent $x$-pass exchanges
these \emph{already pairwise-summed} values across the vertical interfaces, so each of the four copies ends up holding
the total of all four contributions: the partial sums \emph{cascade} around the corner. The same argument applies
recursively in 3D: a vertex DoF shared by $2^3 = 8$ elements receives its full sum after the $z$, $y$, and $x$ passes,
and an edge DoF shared by four elements after the two passes perpendicular to it. Sequentiality of the passes is
essential. Each pass must observe the partial sums produced by the previous one, but only $d$ kernel launches are
required in total, independent of the polynomial degree and of the block size.

At runtime, we execute these passes using specialized kernels that permit fine-grained control over the memory
hierarchy. The kernel maps thread blocks to batches of adjacent element interfaces. Because the block layout is
uniform, extraction and injection are performed using statically unrolled loops and hard-coded face offsets, ensuring
coalesced memory accesses. Furthermore, by processing one axis at a time and ensuring each interface pair is
exclusively handled by exactly one thread, we safely update the boundary values without requiring hardware atomic
operations. This translates a latency-bound graph-traversal problem into a throughput-bound streaming operation. By the
end of this structured phase, all internal block degrees of freedom have reached their final assembled values, while
the block boundary degrees of freedom contain partial sums representing their local block's contribution.

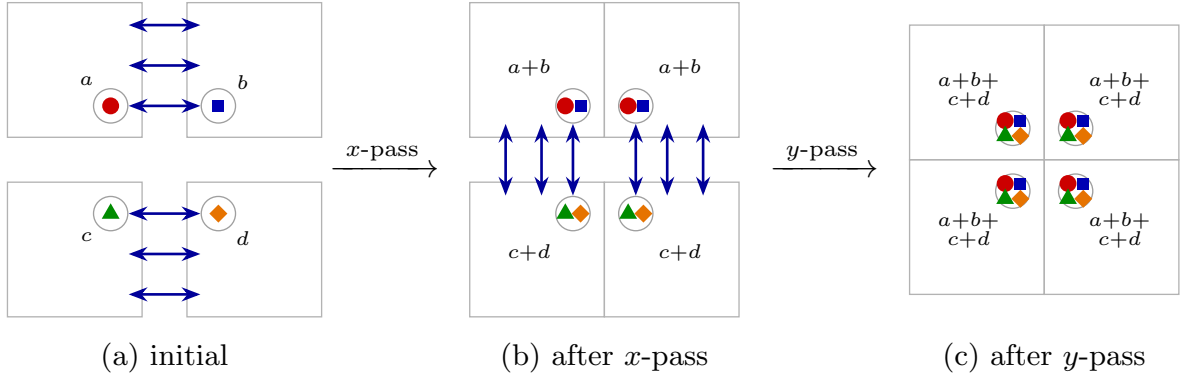
\begin{figure}[htbp]
    \centering
\usetikzlibrary{arrows.meta, positioning, calc, shapes.geometric}
\scalebox{1.25}{%
    \begin{tikzpicture}[
        scale=0.95,
        every node/.style={font=\small},
        gridline/.style={gray!60, thin},
        ring/.style={draw=gray!75, thin, fill=white},
        exch/.style={{Stealth[length=2mm]}-{Stealth[length=2mm]}, thick, blue!60!black},
        mA/.style={circle, fill=red!80!black, inner sep=1.8pt},
        mB/.style={rectangle, fill=blue!70!black, inner sep=1.8pt},
        mC/.style={regular polygon, regular polygon sides=3, fill=green!55!black, inner sep=1.1pt},
        mD/.style={diamond, fill=orange!90!black, inner sep=1.4pt},
        ]

        \def\mo{0.085} 
        \newcommand{\setone}[2]{\node[#2] at #1 {};}
        \newcommand{\settwo}[3]{%
            \node[#2] at ($#1+(-\mo,0)$) {};
            \node[#3] at ($#1+(\mo,0)$) {};}
        \newcommand{\setfour}[1]{%
            \node[mA] at ($#1+(-\mo,\mo)$) {};
            \node[mB] at ($#1+(\mo,\mo)$) {};
            \node[mC] at ($#1+(-\mo,-\mo)$) {};
            \node[mD] at ($#1+(\mo,-\mo)$) {};}

        \newcommand{\lbls}[4]{%
            \node[font=\scriptsize] at ($(nw)+(-0.49,0.43)$) {$\substack{#1}$};
            \node[font=\scriptsize] at ($(ne)+(0.49,0.43)$)  {$\substack{#2}$};
            \node[font=\scriptsize] at ($(sw)+(-0.49,-0.43)$) {$\substack{#3}$};
            \node[font=\scriptsize] at ($(se)+(0.49,-0.43)$)  {$\substack{#4}$};}
        \newcommand{\lblsc}[4]{%
            \node[font=\scriptsize] at ($(nw)+(-0.27,0.27)$) {$\substack{#1}$};
            \node[font=\scriptsize] at ($(ne)+(0.27,0.27)$)  {$\substack{#2}$};
            \node[font=\scriptsize] at ($(sw)+(-0.27,-0.27)$) {$\substack{#3}$};
            \node[font=\scriptsize] at ($(se)+(0.27,-0.27)$)  {$\substack{#4}$};}

        \newcommand{\dsscells}[3]{%
            \pgfmathsetmacro\gx{#1/2}%
            \pgfmathsetmacro\gy{#2/2}%
            \begin{scope}[shift={(-\gx,\gy)}]   
                \draw[gridline] (0,1.5) rectangle (1.5,3); \coordinate (nw) at (1.15,1.85);
                \node[ring, circle, minimum size=3.3mm] at (nw) {};
            \end{scope}
            \begin{scope}[shift={(\gx,\gy)}]    
                \draw[gridline] (1.5,1.5) rectangle (3,3); \coordinate (ne) at (1.85,1.85);
                \node[ring, circle, minimum size=3.3mm] at (ne) {};
            \end{scope}
            \begin{scope}[shift={(-\gx,-\gy)}]  
                \draw[gridline] (0,0) rectangle (1.5,1.5); \coordinate (sw) at (1.15,1.15);
                \node[ring, circle, minimum size=3.3mm] at (sw) {};
            \end{scope}
            \begin{scope}[shift={(\gx,-\gy)}]   
                \draw[gridline] (1.5,0) rectangle (3,1.5); \coordinate (se) at (1.85,1.15);
                \node[ring, circle, minimum size=3.3mm] at (se) {};
            \end{scope}
            \node[align=center, font=\footnotesize] at (1.5,-0.7) {#3};}

        \begin{scope}[shift={(0,0)}]
            \dsscells{0.5}{0.5}{(a) initial}
            \setone{(nw)}{mA} \setone{(ne)}{mB} \setone{(sw)}{mC} \setone{(se)}{mD}
            \lblsc{a}{b}{c}{d}
            \foreach \y in {2.1,2.55,3.0} \draw[exch] (1.1,\y) -- (1.9,\y);
            \foreach \y in {0.9,0.45,0.0} \draw[exch] (1.1,\y) -- (1.9,\y);
        \end{scope}

        \node at (3.95,1.5) {$\xrightarrow{\;x\text{-pass}\;}$};

        \begin{scope}[shift={(4.9,0)}]
            \dsscells{0}{0.5}{(b) after $x$-pass}
            \settwo{(nw)}{mA}{mB} \settwo{(ne)}{mA}{mB}
            \settwo{(sw)}{mC}{mD} \settwo{(se)}{mC}{mD}
            \lbls{a{+}b}{a{+}b}{c{+}d}{c{+}d}
            \foreach \x in {0.4,0.8,1.15} \draw[exch] (\x,1.9) -- (\x,1.1);
            \foreach \x in {1.85,2.2,2.6} \draw[exch] (\x,1.9) -- (\x,1.1);
        \end{scope}

        \node at (8.85,1.5) {$\xrightarrow{\;y\text{-pass}\;}$};

        \begin{scope}[shift={(9.8,0)}]
            \dsscells{0}{0}{(c) after $y$-pass}
            \setfour{(nw)} \setfour{(ne)} \setfour{(sw)} \setfour{(se)}
            \lbls{a{+}b{+}\\c{+}d}{a{+}b{+}\\c{+}d}{a{+}b{+}\\c{+}d}{a{+}b{+}\\c{+}d}
        \end{scope}

    \end{tikzpicture}%
}
    \caption{Dimensionally-split DSS propagation in 2D, shown for a vertex DoF shared by four elements. (a) Before
    assembly each element stores its own partial contribution $a,b,c,d$. (b) The pass along the $x$-axis exchanges and
    sums the copies across faces normal to $x$, so each horizontal neighbour pair holds $a{+}b$ and $c{+}d$. (c) The
    pass along the $y$-axis exchanges the already pairwise-summed values, cascading the partial sums around the corner
    so that every copy holds the total $a{+}b{+}c{+}d$ of all four contributions. The vertex is assembled without
    ever being addressed explicitly.}
    \label{fig:dss_propagation}
\end{figure}

\begin{algorithm}
    \caption{Dimensionally-Split Structured Summation}
    \label{alg:struct_dss}
    \begin{algorithmic}[1]
        \Require Cell-wise residual vector $u \in \mathbb{V}_{\text{cell}}^*$.
        \Ensure Vector $\bar{v}$ with updated internal structured block faces.
        \State $\bar{v} \gets u$
        \For{axis $\alpha \in \{1, 2, 3\}$ (corresponding to $z, y, x$)} \Comment{Exactly three sequential runs in 3D}
        \State \textbf{Parallel Kernel Launch} over all structured internal interfaces $\mathcal{I}_\alpha$ normal to $\alpha$:
        \For{\textbf{each} interface pair $(K_L, K_R) \in \mathcal{I}_\alpha$ \textbf{in parallel}}
        \State Extract values: $v_L \gets \bar{v}|_{K_L, \text{face}_\alpha=p}, \quad v_R \gets \bar{v}|_{K_R, \text{face}_\alpha=0}$
        \State \text{Sum interface contributions:} $s \gets v_L + v_R$
        \State \text{Coalesced Store:} $\bar{v}|_{K_L, \text{face}_\alpha=p} \gets s, \quad \bar{v}|_{K_R, \text{face}_\alpha=0} \gets s$
        \EndFor
        \EndFor
        \State \Return $\bar{v}$
    \end{algorithmic}
\end{algorithm}

\subsection{Generalization to Block-Structured Grids}
\label{sec:dss_unstructured}

The dimensionally-split cascade of Section~\ref{sec:dss_structured} relies on every interior vertex having the regular
$2^d$-element neighborhood of a tensor-product grid, so that the axis passes enumerate exactly the right exchange
partners. At the interfaces \emph{between} macro-blocks this regularity is lost: a macro-vertex of an unstructured base
grid may be shared by any number of blocks (e.g.\ three or five in 2D, Figure~\ref{fig:dss_unstructured}), and no
sequence of axis-aligned pairwise exchanges visits all of its copies. These interfaces must therefore be resolved by
dedicated kernels that traverse the macro-topology explicitly. Their cost is minor, as the affected DoFs form a
lower-dimensional skeleton of the mesh, but their correct and divergence-free execution requires care.

To prevent race conditions and branch divergence on the GPU, the unstructured interfaces are partitioned by codimension
into three mutually disjoint sets, each processed by its own kernel launch: \emph{Faces} (macro-block quads, or in 2D
the codimension-one segments, with their bounding lines and corners masked out), \emph{Lines} (macro-block edges with
their endpoints masked out, in 3D only), and \emph{Vertices} (macro-block corners). The exclusive masking, where each
kernel touches only the strict interior of its entity, iterating the open index range $1,\dots,p-1$ along the entity
and leaving the bounding sub-entities to the lower-dimensional passes, guarantees that the three launches write to
non-overlapping DoFs, so they may run in any order without atomics or inter-launch synchronization. This inter-block
phase operates on the macro-interface skeleton and is disjoint from the structured intra-block phase; it runs after the
structured sweeps, consuming the partial block-boundary sums they produce and completing them into the globally
assembled values.

Faces differ from the lower-dimensional entities in valence: a codimension-one interface is shared by exactly two
blocks, so the Face kernel is a fixed pairwise exchange (cells $a$ and $b$) rather than a variable-length gather. The
only complication is orientation: unstructured mesh generators can connect the two blocks with arbitrary relative
rotations or reflections, so before summing, side $b$ must be re-indexed into the local frame of side $a$. This is
resolved on-the-fly with pure integer arithmetic, with no connectivity lookup: the relative orientation is packed into
a small bitfield (in 3D three logical flags indicating an axis swap and flips along the first or second local
coordinate) from which the matching index on side $b$ is computed algebraically (Figure~\ref{fig:dss_unstructured}).
The kernel then loads both copies, sums them, and stores the result back to each side under its own indexing.

Lines and Vertices, by contrast, carry a runtime valence: a macro-edge or macro-corner of an irregular base grid may be
shared by any number of blocks. Their kernels gather the sharing set into a thread register, accumulate, and scatter
the completed sum back to every copy; these are the only kernels in the entire DSS pipeline that consult a sharing list
rather than a fixed-arity neighbor map. The sharing list is padded to a fixed width so the launch stays branch-free,
and (like the Faces) each copy is re-indexed into a common canonical frame before accumulation. A one-dimensional
reversal flag for lines, no orientation for vertices.

A subtlety distinguishes these entities from the structured and Face passes and is in fact the crux of the
block-structured generalization. A macro-corner (or, in 3D, a macro-edge) is in general shared by \emph{several cells
    of the same block}, and the structured pass has already summed those into a common per-block value that it replicates
on every within-block copy. Adding all copies again across blocks would therefore multiply the result by the
within-block multiplicity. The gather thus sums only \emph{one representative copy per block} and scatters the total
back to all copies, representatives and non-representatives alike. Concretely, take the midpoint $X$ of a shared
macro-edge in Figure~\ref{fig:dss_unstructured}: it is the corner at which two sub-cell faces meet, hence a
\emph{Vertex} entity, and in 2D it is shared by four cells (two in each of the two adjacent blocks). Initialize every
copy to $1$. The structured pass sums the two within-block copies on each side, leaving $2$ on all four. A naive gather
over all four copies would return $2+2+2+2 = 8$; instead the Vertex kernel adds one representative per block, $2 + 2 =
    4$, and writes $4$ to every copy, which is the correctly assembled value. The number of contributions to such an entity
is therefore its number of sharing \emph{blocks}, not its number of cell copies: two for the mid-edge node (a pairwise
exchange), and three or five for the central macro-corner of Figure~\ref{fig:dss_unstructured}. The same reasoning
applies verbatim to the 3D Line kernel, where a macro-edge interior is likewise reached by several cells of each block;
the representative selection carries over unchanged.

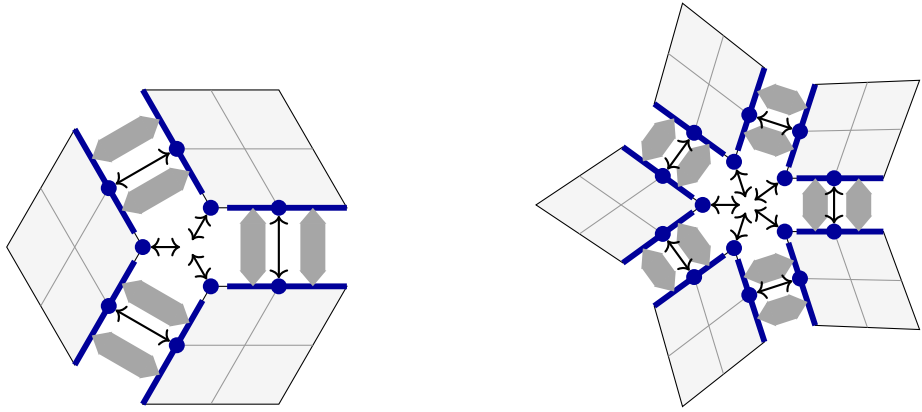
\begin{figure}[htbp]
    \centering
\usetikzlibrary{arrows.meta, calc}

\def\inset{0.10}
\def\gap{0.6}
\def\fgap{0.12}   
\def\dotrad{3pt}  
\def\midgap{0.07cm}  
\newlength{\arrowthick}\setlength{\arrowthick}{8pt}   
\newlength{\arrowtip}\setlength{\arrowtip}{10pt}       

\newcommand{\facelink}[3]{%
\coordinate (FEa) at ($($(#1)!{\inset}!(0,0)$)+(#2:\gap)$);
\coordinate (FEb) at ($($(#1)!{\inset}!(0,0)$)+(#3:\gap)$);
\pgfmathsetmacro{\mf}{(1+\inset)/2}
\coordinate (FLa) at ($($(#1)!{\mf}!(0,0)$)+(#2:\gap)$);
\coordinate (FLb) at ($($(#1)!{\mf}!(0,0)$)+(#3:\gap)$);
\draw[facelink] ($(FLa)!0.5!(FEa)$) -- ($(FLb)!0.5!(FEb)$);
\coordinate (FOoa) at (#2:\gap);
\coordinate (FOob) at (#3:\gap);
\coordinate (FGa) at ($(FOoa)!{\fgap}!(FEa)$);
\coordinate (FGb) at ($(FOob)!{\fgap}!(FEb)$);
\draw[facelink] ($(FGa)!0.5!(FLa)$) -- ($(FGb)!0.5!(FLb)$);
\draw[cornerlink] (FLa) -- (FLb);
\fill[blue!60!black] (FLa) circle (\dotrad);
\fill[blue!60!black] (FLb) circle (\dotrad);}

\newcommand{\ucell}[4]{%
        \coordinate (Oo) at (#1:\gap);
        \coordinate (Pa) at ($(#2)+(#1:\gap)$);
        \coordinate (Pb) at ($(#3)+(#1:\gap)$);
        \coordinate (Pc) at ($(#4)+(#1:\gap)$);
        \coordinate (Va) at ($(Pa)!\inset!(Oo)$);
        \coordinate (Vb) at ($(Pb)!\inset!(Oo)$);
        \coordinate (Vc) at ($(Pc)!\inset!(Oo)$);
        \coordinate (Moa) at ($(Oo)!0.5!(Va)$);
        \coordinate (Mab) at ($(Va)!0.5!(Vb)$);
        \coordinate (Mbc) at ($(Vb)!0.5!(Vc)$);
        \coordinate (Mco) at ($(Vc)!0.5!(Oo)$);
        \draw[cell] (Oo) -- (Va) -- (Vb) -- (Vc) -- cycle;
        \draw[subdiv] (Moa) -- (Mbc);
        \draw[subdiv] (Mab) -- (Mco);
        \coordinate (FSa) at ($(Oo)!{\fgap}!(Va)$);
        \coordinate (FSc) at ($(Oo)!{\fgap}!(Vc)$);
        \draw[face] (FSa) -- ($(Moa)!\midgap!(FSa)$);
        \draw[face] ($(Moa)!\midgap!(Va)$) -- (Va);
        \draw[face] (FSc) -- ($(Mco)!\midgap!(FSc)$);
        \draw[face] ($(Mco)!\midgap!(Vc)$) -- (Vc);
        \fill[blue!60!black] (Oo) circle (\dotrad);}

\def\R{2.0}

\begin{tikzpicture}[
        scale=1.0,
        cell/.style={draw=black, thin, fill=gray!8},
        face/.style={blue!60!black, line width=2.2pt},
        subdiv/.style={black!40, line width=0.4pt},
        facelink/.style={gray!70, double=gray!70, double distance=\the\arrowthick,
        line width=0.7pt,
        {Triangle[width=\the\arrowtip,length=6pt]}-{Triangle[width=\the\arrowtip,length=6pt]}},
        cornerlink/.style={black, <->, line width=0.8pt,
                        shorten >=\dotrad, shorten <=\dotrad},
        ]
        \coordinate (O)  at (0,0);
        \coordinate (H0) at (0:\R);
        \coordinate (H1) at (60:\R);
        \coordinate (H2) at (120:\R);
        \coordinate (H3) at (180:\R);
        \coordinate (H4) at (240:\R);
        \coordinate (H5) at (300:\R);

        \ucell{60}{H0}{H1}{H2}
        \ucell{180}{H2}{H3}{H4}
        \ucell{300}{H4}{H5}{H0}

        \facelink{H0}{60}{300}
        \facelink{H2}{60}{180}
        \facelink{H4}{180}{300}

        \draw[cornerlink] (60:\gap) -- (0,0);
        \draw[cornerlink] (180:\gap) -- (0,0);
        \draw[cornerlink] (300:\gap) -- (0,0);

\end{tikzpicture}
\hspace{2.2cm}
\begin{tikzpicture}[
        scale=1.0,
        cell/.style={draw=black, thin, fill=gray!8},
        face/.style={blue!60!black, line width=2.2pt},
        subdiv/.style={black!40, line width=0.4pt},
        facelink/.style={gray!70, double=gray!70, double distance=\the\arrowthick,
        line width=0.7pt,
        {Triangle[width=\the\arrowtip,length=6pt]}-{Triangle[width=\the\arrowtip,length=6pt]}},
        cornerlink/.style={black, <->, line width=0.8pt,
                        shorten >=\dotrad, shorten <=\dotrad},
        ]
        \def\Rf{1.45}   
        \def\Rt{2.45}   
        \coordinate (O)  at (0,0);
        \coordinate (H0) at (0:\Rf);
        \coordinate (H1) at (36:\Rt);
        \coordinate (H2) at (72:\Rf);
        \coordinate (H3) at (108:\Rt);
        \coordinate (H4) at (144:\Rf);
        \coordinate (H5) at (180:\Rt);
        \coordinate (H6) at (216:\Rf);
        \coordinate (H7) at (252:\Rt);
        \coordinate (H8) at (288:\Rf);
        \coordinate (H9) at (324:\Rt);

        \ucell{36}{H0}{H1}{H2}
        \ucell{108}{H2}{H3}{H4}
        \ucell{180}{H4}{H5}{H6}
        \ucell{252}{H6}{H7}{H8}
        \ucell{324}{H8}{H9}{H0}

        \facelink{H0}{36}{324}
        \facelink{H2}{36}{108}
        \facelink{H4}{108}{180}
        \facelink{H6}{180}{252}
        \facelink{H8}{252}{324}

        \draw[cornerlink] (36:\gap) -- (0,0);
        \draw[cornerlink] (108:\gap) -- (0,0);
        \draw[cornerlink] (180:\gap) -- (0,0);
        \draw[cornerlink] (252:\gap) -- (0,0);
        \draw[cornerlink] (324:\gap) -- (0,0);

\end{tikzpicture}
    \caption{Unstructured DSS exchange at macro-vertices of irregular valence in 2D, here a three-block (left) and a five-block (right) configuration. Blue dots mark the shared macro-vertex copies and
        blue dashed segments the macro-block edges meeting at it; in 2D these codimension-one edges are the
        \emph{Face} entities, while Line entities appear only in 3D. Black arrows depict the
        \emph{Vertex}-kernel exchange: no sequence of axis-aligned pairwise exchanges visits all copies of such a vertex,
        so a dedicated kernel gathers the contributions of all sharing macro-blocks into a register, accumulates them,
        and scatters the completed sum back to every participant. The mid-edge nodes at the centers of the macro-block
        edges are also \emph{Vertex} entities: each is shared by two cells \emph{within} every adjacent block, but the
        structured pass has already summed those within-block copies to a common value, so only \emph{one representative
            copy per block} enters the inter-block sum. This is why each such node is drawn with a single dot per block and a
        single pairwise arrow (two blocks, two representatives) rather than four converging spokes; in general the
        number of arrows meeting at an entity equals its number of sharing blocks (here three and five at the central
        macro-corner). Gray arrows along the macro-block edges indicate the
        \emph{Face}-entity exchanges; their differing orientations reflect the per-face canonicalization
        (extract$\rightarrow$permute$\rightarrow$accumulate$\rightarrow$inverse-permute) that reconciles the local orientations of neighboring blocks,
        computed from their vertex ordering.}
    \label{fig:dss_unstructured}
\end{figure}

\begin{algorithm}
    \caption{Unstructured Inter-Block Summation}
    \label{alg:unstruct_dss}
    \begin{algorithmic}[1]
        \Statex \textbf{Require:} Vector $\bar{v}$ from structured summation; within-block copies already share a value.
        \Statex \hphantom{\textbf{Require:} }Disjoint sets $\mathcal{F}$ (faces, pairwise), $\mathcal{L}$ (lines, 3D), $\mathcal{V}$ (vertices).
        \Statex \hphantom{\textbf{Require:} }Each line/vertex entity flags one cell per sharing block as its representative.
        \Statex \textbf{Ensure:} Vector $\bar{v}$ with globally assembled macro-block interfaces.

        \For{\textbf{each} macro-face $(K_a, K_b, o) \in \mathcal{F}$ \textbf{in parallel}} \Comment{two blocks per face}
        \State $v_a \gets \bar{v}|_{K_a,\text{face}}$, \quad $v_b \gets \bar{v}|_{K_b,\text{face}}$ \Comment{face interiors}
        \State $s \gets v_a + \textsc{Reindex}(v_b, o)$ \Comment{$o$: orientation, map $b\!\to\!a$}
        \State $\bar{v}|_{K_a,\text{face}} \gets s$, \quad $\bar{v}|_{K_b,\text{face}} \gets \textsc{Reindex}^{-1}(s, o)$
        \EndFor

        \For{\textbf{each} macro-line $L \in \mathcal{L}$ \textbf{in parallel}} \Comment{$r_K$: reversal flag}
        \State $s \gets 0$
        \For{\textbf{each} representative cell $K \in L$ (one per block)}
        \State $s \gets s + \textsc{Permute}_{\text{1D}}(\bar{v}|_{K,\text{line}},\, r_K)$
        \EndFor
        \For{\textbf{each} sharing cell $K \in L$ (all copies)}
        \State $\bar{v}|_{K,\text{line}} \gets \textsc{InversePermute}_{\text{1D}}(s,\, r_K)$
        \EndFor
        \EndFor

        \For{\textbf{each} macro-vertex $V \in \mathcal{V}$ \textbf{in parallel}}
        \State $s \gets 0$
        \For{\textbf{each} representative cell $K \in V$ (one per block)}
        \State $s \gets s + \bar{v}|_{K,\text{vertex}}$
        \EndFor
        \For{\textbf{each} sharing cell $K \in V$ (all copies)}
        \State $\bar{v}|_{K,\text{vertex}} \gets s$
        \EndFor
        \EndFor

        \State \Return $\bar{v}$
    \end{algorithmic}
\end{algorithm}

\section{Implementation}
\label{sec:implementation}


\subsection{Memory layout: the BXYZE format}
\label{sec:layout}

The field data reside in a multi-dimensional array with layout $[B,\, n_x,\, n_y,\, n_z,\, n_e]$ (which we refer to as
the \emph{BXYZE} layout), where $B$ is a fixed block size, $(n_x, n_y, n_z)$ are the local DoF indices within an
element, and $n_e$ enumerates the elements of a block. Because the cell-wise storage decouples the layout from the
global mesh topology, this ordering can be chosen freely to suit the hardware; the choice governs a trade-off between
memory coalescing and element-contiguity, best seen at the two extremes. In an element-major layout (EXYZ), the element
index is the fastest stride: threads operating on the same local DoF across consecutive elements read consecutive
words: coalesced. However, operations that touch only a proper subset of an element's DoFs (e.g.\ the DoFs of a single
face during DSS or face restriction) find these fragments interleaved across distant elements, fragmenting the memory
request pattern. Conversely, in a node-major layout (XYZE) all DoFs of one element are contiguous, so element-local
sub-accesses are dense; but threads working on distinct elements stride by the full per-element footprint $(p+1)^d$,
and coalescing across elements is lost.

The tension is thus between two classes of kernels. For volumetric operator evaluation, where every thread processes
whole elements, EXYZ is the optimal layout: accesses are coalesced and nothing else matters. But the same layout is
counterproductive for the DSS kernels of Section~\ref{sec:dss}, which read and write face, edge, and vertex fragments:
in EXYZ these fragments are scattered with element-sized strides, and the surface kernels' throughput collapses. The
blocked BXYZE layout is the compromise: grouping $B$ spatially adjacent elements into one memory tile keeps accesses
along the fastest axis coalesced for the volumetric kernels, while the per-element data remain compact enough that
interface fragments span few cache lines. (Internal padding where fewer than $B$ valid elements exist is a negligible
concern: with $B \le 128$, at most $B$ wasted cells per block boundary are irrelevant at realistic problem sizes.) The
layout also maps naturally onto sum factorisation: the 1D contractions operate independently along the $B$ and $n_e$
axes and translate directly into vectorised multi-dimensional array operations. We back these claims up empirically in
Section~\ref{sec:results_layout}, where the volumetric and DSS kernels are benchmarked across the layout design space.

\subsection{Kernel implementation in Triton}
\label{sec:triton_kernels}

Our entire solver stack is implemented in \emph{Triton}~\cite{triton2019}, a domain-specific language and compiler for
GPU kernels embedded in Python. For readers from the finite element community unfamiliar with it: Triton occupies a
middle ground between high-level array frameworks (such as JAX/XLA~\cite{jax2018github}) and hand-written CUDA. A
Triton kernel is written as a Python function operating on small multi-dimensional \emph{tiles} (blocks of a tensor)
rather than on individual threads; the programmer specifies \emph{what} each tile-level program instance loads,
computes, and stores, while the compiler decides the thread-level details: register allocation, vectorization of loads,
shared-memory staging, and the mapping of tile operations onto the SIMT hardware. Notably, Triton bypasses the CUDA
toolchain: kernels are compiled through its own MLIR-based pipeline directly to PTX (or, on AMD hardware, to the
corresponding AMDGPU ISA), so no \texttt{.cu} source is ever generated, and the same kernel source runs on both NVIDIA
and AMD GPUs.

This tile-based abstraction matches our cell-wise storage format exactly: a memory tile of the BXYZE layout is a
natural Triton block, and the sum-factorized 1D contractions become tile-level tensor operations. The result is that
kernels which would require hundreds of lines of carefully tuned CUDA are expressed in a few dozen lines of Python,
while remaining competitive with vendor libraries in achieved bandwidth. Because Triton operates on PyTorch tensors,
every kernel can be verified directly against a reference implementation written in high-level, abstract linear algebra
on the same data, which is an invaluable property when developing the element kernels, where reproducing the
sum-factorized tensor contractions by hand is highly error-prone, especially for the intricate tile decompositions of
non-power-of-two polynomial degrees ($p = 4, 5$). We also note that a distributed, multi-GPU extension of the
programming model exists~\cite{zheng2025triton}, offering a natural path to scaling out the present single-device
implementation.

Two implementation aspects deserve explicit mention. First, \emph{shared-memory bank conflicts}: when a tile is staged
in on-chip shared memory, which is divided into 32 banks, threads that simultaneously access different addresses in the
same bank are serialized. Tensor-product index patterns are particularly prone to this, since the stride between
consecutively accessed DoFs is often a multiple of the bank count. Triton's compiler mitigates these conflicts
automatically by \emph{swizzling} (applying a bank-aware permutation to the shared-memory layout of staged
tiles~\cite{triton2019,zhou2026linear}), so the contractions proceed without swizzling tricks customary in hand-written
CUDA kernels (see for example~\cite{cui2025implementation}).

Second, \emph{power-of-two tile shapes}: Triton requires every tile dimension to be a power of two. The per-element DoF
count $(p+1)$ meets this only for $p \in \{1, 3, 7\}$. For $p = 2, 4, 5$, padding a single tile up to the next power of
two would waste a substantial fraction of memory and bandwidth, so our data layout instead decomposes the DoF range of
such elements into several tensors of power-of-two extent (e.g.\ $5 = 4 + 1$), and each kernel processes the resulting
sub-tiles; this complicates the kernels but keeps the stored data dense. For $p = 6$ we make the opposite choice and
simply pad to $8^3$: the utilization $(7/8)^3 \approx 0.67$ is still reasonable, and avoiding the multi-tensor
decomposition removes its structural overhead. All supported degrees $p = 1, \dots, 7$ are thus covered by one of the
two strategies. This padding is the source of the gap between the \emph{theory} and \emph{achieved} TFLOP/s rates for
$p \ge 3$ (cf.\ Table~\ref{tab:perf_teaser_laplace}): the hardware issues operations on the padded tiles, while the
theoretical rate counts only the useful contractions.

\begin{table}
    \centering
\begin{tabular}{l c c c c c c c}
    \toprule
                       & $p=1$ & $p=2$ & $p=3$ & $p=4$ & $p=5$ & $p=6$ & $p=7$ \\
    \midrule
    Warps              & 4     & 4     & 1     & 1     & 1     & 1     & 1     \\
    Block exec ($B_e$) & 128   & 128   & 8     & 4     & 2     & 4     & 4     \\
    Block data ($B$)   & 128   & 128   & 32    & 32    & 32    & 32    & 32    \\
    \bottomrule
\end{tabular}

    \caption{Best kernel configuration per polynomial degree. \emph{Block exec} ($B_e$) is the number of elements processed
        per kernel launch tile; \emph{block data} ($B$) is the storage block size of the $[B,\,n_x,\,n_y,\,n_z,\,n_e]$
        layout.}
    \label{tab:best_config}
\end{table}

Finally, the dominant arithmetic primitive (the 1D contraction of sum factorization) is expressed through Triton's
\texttt{dot} primitive, which the compiler lowers onto the GPU's matrix units (tensor cores). The contractions of all
elements in a tile share the same small 1D matrix ($V$, or its transpose), so they batch naturally into the
matrix-matrix shapes these units consume. The matrix units double the FP64 peak relative to the vector pipeline; it
additionally frees the vector pipeline for address arithmetic and keeps the contraction off the critical path. At low
polynomial degrees the 1D matrices are too small to fill the tensor-core tile shapes that Triton supported out of the
box. I therefore implemented support for the small-tile MMAv2 instruction shape ($8 \times 8 \times 4$) in the Triton
compiler, a contribution that has since been merged
upstream.\footnote{\url{https://github.com/triton-lang/triton/pull/10060}} On the benchmarked matrix-matrix
multiplication shapes, Triton for some sizes slightly outperforms cuBLAS, which we attribute to the compiler's software
pipelining of the main loop, overlapping global-memory staging with tensor-core computation.

It is worth stressing, for readers from the finite element community, what the implementation does \emph{not} contain.
There is no global degree-of-freedom numbering and no element-to-global index map anywhere in the solver: every kernel
reaches its operands through statically known offsets into the $[B,\, n_x,\, n_y,\, n_z,\, n_e]$ tiles. The only
connectivity data that survive are the macro-block adjacency lists consulted by the unstructured DSS vertex- and
line-kernels of Section~\ref{sec:dss_unstructured}, and these reference merely the lower-dimensional skeleton of the
mesh. The gather-scatter index arrays that consume both storage and bandwidth in classical matrix-free frameworks thus
have no counterpart in our data structures.


\section{Numerical Results}
\label{sec:results}

Throughout this section we report throughput in giga-degrees-of-freedom per second (GDoF/s), counting each degree of
freedom once. This normalizes the cost to the size of the problem actually being solved rather than to an
implementation-dependent count of local element nodes, and it makes rates directly comparable across polynomial degrees
and across operators. The unique-DoF count itself is obtained in one of two ways. On structured Cartesian grids it is
known in closed form, $N_\text{unique} = (n p + 1)^d$ for $n$ cells per direction, and the reported rates are exact. On
general (unstructured or $h$-adaptively refined) meshes we instead measure the cell-wise rate (total stored DoFs
divided by apply time) and rescale it to unique DoFs by the large-mesh factor $\big(p/(p+1)\big)^d$. This factor is the
asymptotic interior ratio and ignores the surface DoFs of the domain boundary, which are shared by fewer elements; it
therefore slightly \emph{over}counts the redundancy and \emph{under}counts the unique DoFs, so the resulting unique
rates are a conservative lower bound on the true throughput.

One further caveat applies specifically to the DSS rates: for $p > 1$ only the interface DoFs shared between elements
take part in the direct stiffness summation (the interior DoFs of each element are never exchanged), so the
per-unique-DoF DSS throughput is computed against the full DoF count even though a shrinking fraction of those DoFs is
actually touched as $p$ grows. This is deliberate: it keeps the DSS rate on the same per-DoF footing as the volumetric
rates, so the two can be read on a common axis, and it is the reason the apparent DSS throughput \emph{rises} with $p$
even though the work per interface DoF does not.

\subsection{Layout study: the volumetric--DSS trade-off}
\label{sec:results_layout}

We begin by substantiating the central layout claim of Section~\ref{sec:layout}: that the block size $B$ of the $[B,\,
            n_x,\, n_y,\, n_z,\, n_e]$ layout governs an irreconcilable tension between volumetric operator evaluation and the
surface DSS kernels, and that an intermediate $B$ is the right compromise. The measurements are performed on a
Cartesian grid obtained by uniform hierarchical refinement of a single coarse cell ($16.7$ million active elements,
$\approx 1.07$ billion active DoFs at $p=3$); on this structured geometry every macro-interface is conforming and
regular, so the reported DSS figures exercise the structured-path kernels of Section~\ref{sec:dss_structured} alone,
isolating them from the unstructured and hanging-node machinery. For each polynomial degree we sweep $B$ over powers of
two and, at every $B$, retain the launch configuration that maximizes volumetric throughput, reporting the
corresponding mass-operator volumetric and DSS rates per unique DoF.

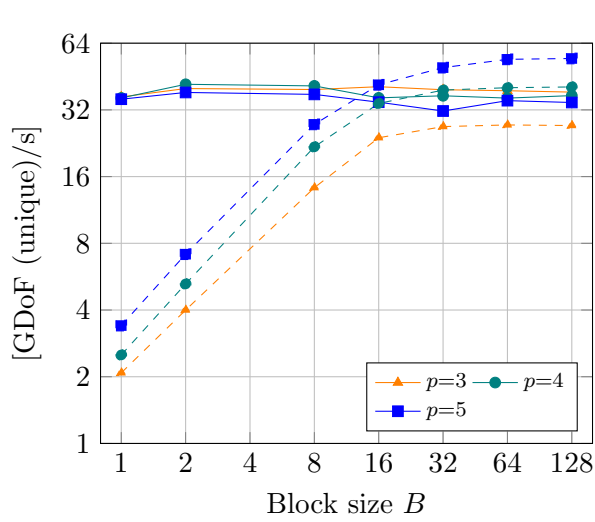
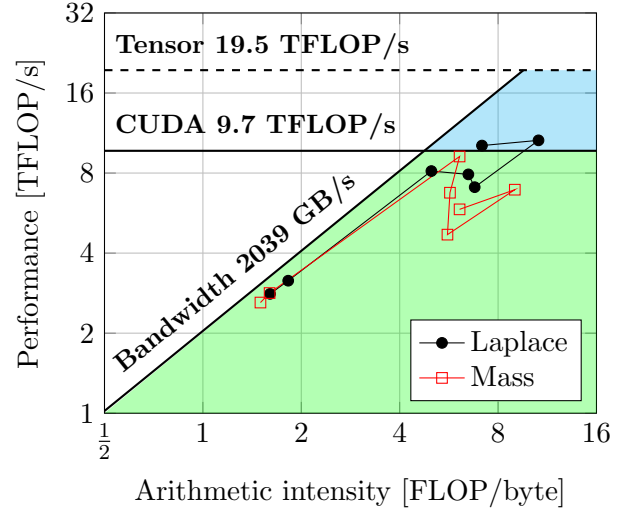
\begin{figure}
    \centering
    \begin{subfigure}[b]{0.49\linewidth}
        \centering
\begin{tikzpicture}
    \begin{loglogaxis}[
            log basis x=2,
            log basis y=2,
            width=\linewidth,
            height=0.85\linewidth,
            xlabel={Block size $B$},
            ylabel={ [GDoF (unique)/s]},
            xmin=0.8, xmax=160,
            ymin=1, ymax=64,
            xtick={1,2,4,8,16,32,64,128},
            xticklabels={1,2,4,8,16,32,64,128},
            ytick={1,2,4,8,16,32,64},
            yticklabels={1,2,4,8,16,32,64},
            grid=both,
            legend style={font=\scriptsize, at={(0.97,0.03)}, anchor=south east},
            legend columns=2,
            legend cell align=left,
        ]
        \addplot[orange, mark=triangle*] table[col sep=comma, x index=0, y index=6] {raw_results/sweep_p3.csv};
        \addplot[teal,   mark=*]         table[col sep=comma, x index=0, y index=6] {raw_results/sweep_p4.csv};
        \addplot[blue,   mark=square*]   table[col sep=comma, x index=0, y index=6] {raw_results/sweep_p5.csv};
        \addplot[orange, mark=triangle*, dashed] table[col sep=comma, x index=0, y index=7] {raw_results/sweep_p3.csv};
        \addplot[teal,   mark=*,         dashed] table[col sep=comma, x index=0, y index=7] {raw_results/sweep_p4.csv};
        \addplot[blue,   mark=square*,   dashed] table[col sep=comma, x index=0, y index=7] {raw_results/sweep_p5.csv};
        \legend{$p{=}3$, $p{=}4$, $p{=}5$}
    \end{loglogaxis}
\end{tikzpicture}
        \caption{Block-size sweep (structured path). Solid lines: volumetric throughput; dashed lines: DSS throughput.}
        \label{fig:dss_layout_sweep}
    \end{subfigure}
    \hfill
    \begin{subfigure}[b]{0.49\linewidth}
        \centering
\begin{tikzpicture}
    \begin{loglogaxis}[
            width=\linewidth,
            height=0.85\linewidth,
            xlabel={Arithmetic intensity [FLOP/byte]},
            ylabel={Performance [TFLOP/s]},
            log basis x=2,
            xmin=0.5, xmax=16,
            xtick={0.5,1,2,4,8,16},
            xticklabels={$\frac{1}{2}$,1,2,4,8,16},
            log basis y=2,
            ymin=1, ymax=32,
            ytick={1,2,4,8,16,32},
            yticklabels={1,2,4,8,16,32},
            grid=both,
            legend pos=south east,
            legend cell align=left,
        ]
        \addplot[green, fill=green, draw=none, fill opacity=0.3, forget plot]
        coordinates {(0.5,1) (4.76,9.7) (16,9.7) (16,1) (0.5,1)};
        \addplot[cyan, fill=cyan, draw=none, fill opacity=0.3, forget plot]
        coordinates {(4.76,9.7) (9.56,19.5) (16,19.5) (16,9.7) (4.76,9.7)};

        \addplot[thick, domain=0.5:9.56, samples=2, forget plot] {2.039*x}
        node[pos=0.35, sloped, above, font=\small\bfseries] {Bandwidth 2039 GB/s};
        \addplot[thick, domain=0.5:16, samples=2, forget plot] {9.7}
        node[pos=0.0, above right, font=\small\bfseries] {CUDA 9.7 TFLOP/s};
        \addplot[thick, dashed, domain=0.5:16, samples=2, forget plot] {19.5}
        node[pos=0.0, above right, font=\small\bfseries] {Tensor 19.5 TFLOP/s};

        \addplot[black, mark=*, mark size=2pt]
        table[col sep=comma, x=ai, y=tflops] {raw_results/laplace_roofline.csv};
        \addlegendentry{Laplace}
        \addplot[red, mark=square, mark options={fill=none}, mark size=2pt]
        table[col sep=comma, x=ai, y=tflops] {raw_results/mass_roofline.csv};
        \addlegendentry{Mass}
    \end{loglogaxis}
\end{tikzpicture}
        \caption{Roofline, best configuration.}
        \label{fig:roofline}
    \end{subfigure}
    \caption{(a) Volumetric (solid) and DSS (dashed) throughput of the mass operator in cell-wise storage as a function
        of the block size $B$, for polynomial degrees $p = 3, 4, 5$, measured on a Cartesian grid of $\approx 1.07$
        billion active DoFs (A100 80\,GB SXM). Rates are per unique DoF; at each $B$ the configuration maximizing
        volumetric throughput is shown. Volumetric throughput is nearly flat in $B$ while DSS throughput rises steeply
        and saturates past $B \approx 32$, so an intermediate $B = 16$--$32$ is the compromise. (b) Roofline of the
        mass operator~\eqref{eq:mass_operator} and the Laplace operator~\eqref{eq:laplace_operator} in the best
        configuration; solid lines are
        the memory-bandwidth slope (2\,039\,GB/s) and the vector FP64 ceiling (9.7\,TFLOP/s), the dashed line the
        tensor-core FP64 ceiling (19.5\,TFLOP/s).}
    \label{fig:dss_layout}
\end{figure}

\begin{figure}
    \centering
%
\pgfplotsset{
    bp1unique/.style={
            log basis x=10,
            width=\linewidth, height=0.85\linewidth, xlabel={Unique DoFs}, ylabel={GDoF (unique) /s}, xmin=20, xmax=1e9, ymin=0,
            ymax=45, grid=both, legend style={font=\scriptsize, at={(0.03,0.97)}, anchor=north west}, legend columns=2, legend cell
            align=left, }, } \newcommand{\bpuniqueplot}[1]{%
    \begin{tikzpicture}
        \begin{semilogxaxis}[bp1unique]
            \addplot[red,    mark=*]         table[col sep=comma, x index=0,  y index=1]  {#1};
            \addplot[orange, mark=square*]   table[col sep=comma, x index=2,  y index=3]  {#1};
            \addplot[teal,   mark=triangle*] table[col sep=comma, x index=4,  y index=5]  {#1};
            \addplot[blue,   mark=diamond*]  table[col sep=comma, x index=6,  y index=7]  {#1};
            \addplot[violet, mark=pentagon*] table[col sep=comma, x index=8,  y index=9]  {#1};
            \addplot[brown,  mark=otimes*]   table[col sep=comma, x index=10, y index=11] {#1};
            \addplot[black,  mark=x]         table[col sep=comma, x index=12, y index=13] {#1};
            \legend{$p{=}1$, $p{=}2$, $p{=}3$, $p{=}4$, $p{=}5$, $p{=}6$, $p{=}7$}
        \end{semilogxaxis}
    \end{tikzpicture}%
}
\begin{subfigure}[b]{0.49\linewidth}
    \centering
    \bpuniqueplot{raw_results/BP1_cartesian_plot.csv}
    \caption{Cartesian.}
    \label{fig:bp1_unique_cartesian}
\end{subfigure}
\hfill
\begin{subfigure}[b]{0.49\linewidth}
    \centering
    \bpuniqueplot{raw_results/BP1_mappingQ_plot.csv}
    \caption{ General geometry.}
    \label{fig:bp1_unique_mappingq}
\end{subfigure}
    \caption{BP1 mass-operator throughput per unique DoF as a function of the number of unique DoFs, for polynomial
        degrees $p = 1,\dots,7$ on a structured Cartesian grid (A100 80\,GB SXM, \texttt{fp64}). Each marker corresponds to a
        global refinement level. (a) Cartesian mapping; (b) General mapping. Throughput per unique
        DoF saturates once the problem fills the device, the achievable rate decreasing with $p$ as the cell-wise
        duplication factor $\big(n(p{+}1)/(np{+}1)\big)^3$ grows.}
    \label{fig:bp1_unique}
\end{figure}
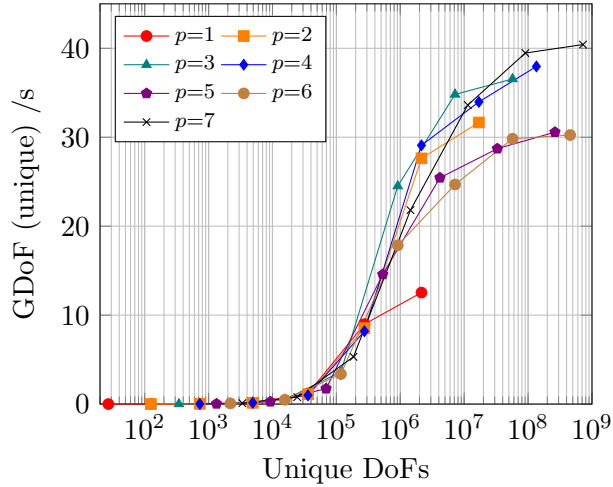
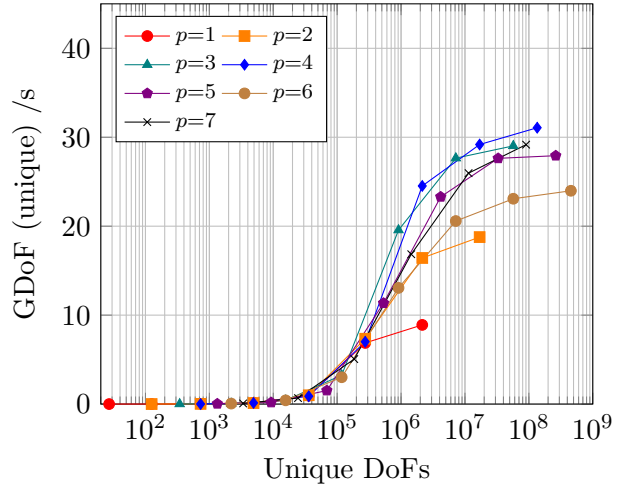

Figure~\ref{fig:dss_layout}(a) lays the trade-off bare. The solid curves (volumetric throughput) are essentially flat
in $B$: whole-element streaming is already coalesced once $B \ge 2$, so enlarging the block buys nothing and, through
padding and slightly worse cache behaviour, costs a few percent at the largest sizes. The dashed curves (DSS
throughput) behave in the opposite way: at $B=1$ each structured face exchange straddles a block boundary and the
kernel is starved, but it climbs steeply and monotonically with $B$ as more interfaces become intra-block and the axis
sweeps run on dense, coalesced tiles, saturating only past $B \approx 32$--$64$, an order-of-magnitude gain from the
smallest to the largest block. The two curves cross around $B \approx 16$--$32$: below it the surface kernels dominate
the cost, above it they are essentially free while the volumetric rate has begun to erode. A block size of $B =
    16$--$32$ therefore captures almost all of the DSS speed-up at a negligible volumetric penalty, which is precisely the
compromise adopted throughout the remainder of this work. Notably, the DSS rate \emph{improves} with polynomial degree
(the surface-to-volume ratio of the exchanged DoFs falls as $p$ grows) so the trade-off only relaxes for the
higher-order elements that are the main target of the method.

\subsection{DSS throughput dependence on mesh structure}
\label{sec:results_dss_struct_unstruct}

We now isolate the cost of the direct stiffness summation itself and ask how much the structured (dimension-split) path
of Section~\ref{sec:dss_structured} gains over the general unstructured path. Both are exercised on the same geometry
(a unit cube under uniform refinement) so that the meshes are identical and only the code path differs: the structured
path exploits the tensor-product regularity of the interfaces, whereas the unstructured path treats every shared DoF
through the general indirection machinery. DSS carries essentially no arithmetic; each shared DoF is gathered, summed,
and scattered. Its arithmetic intensity is therefore negligible and the kernel is memory bound, so we report only the
achieved DRAM traffic and, alongside the GDoF/s rate, its speed-of-light (SoL): the fraction of peak memory bandwidth
attained. Table~\ref{tab:dss_structured_vs_unstructured} collects both metrics for $p = 1, \dots, 7$.

Two trends stand out. First, the structured path sustains a high and roughly constant bandwidth utilization (around
$80\%$ of peak across all degrees), while the unstructured path plateaus markedly lower, near $60\%$, the gap
reflecting the scattered, less coalesced memory accesses of the general indirection. Second, the throughput advantage
of the structured path \emph{widens} with polynomial degree: at $p = 1$ the two paths deliver essentially the same
GDoF/s rate, but by the highest degrees the structured path is substantially faster. The $p = 1$ tie is instructive.
There the two paths run genuinely different algorithms: the structured pass streams every vertex of the Cartesian grid
a fixed number of times by construction, which inflates its traffic and hence depresses its SoL even though it moves
the same number of unique DoFs per second as the unstructured pass. Their equal throughput at $p=1$ is thus a
coincidence of two different cost structures meeting, not a sign that the paths are equivalent; as soon as $p > 1$ and
interior DoFs dilute the interface work, the structured path's regularity pulls ahead.

Table~\ref{tab:combined_operator} folds the two stages into the throughput of the \emph{entire} assembled mass operator
(one volumetric apply followed by a complete DSS, i.e.\ a single matrix--vector product of the continuous operator $C^T
    A C$. The volumetric cost is fixed by the geometry (Cartesian or general/curved) and the DSS cost by the path
(structured or unstructured), so the four rows bracket the realistic operating range. Throughput rises monotonically
with $p$, from $3.0\,$GDoF/s at $p=1$, where the cell-wise redundancy is most punishing, to $27.0\,$GDoF/s at $p=7$ on
the Cartesian/structured path, as the surface-to-volume ratio of the exchanged DoFs falls and the coalesced volumetric
apply comes to dominate the increasingly cheap DSS. The bottom row lists the libCEED BP1 mass-operator
throughput~\cite{osti_1822610} on the same A100 80\,GB SXM device. Since libCEED's BP1 is posed on a general (curved)
geometry, the most directly comparable configuration is our general-geometry volume resolved through the structured DSS
path; against it our entire operator is faster per unique DoF throughout, from $1.6\times$ at $p=1$ to $4.4\times$ at
$p=5$, the margin widening with degree. This is the direct dividend of removing indirect gather--scatter from the
operator: the volumetric apply streams coalesced cell-wise data and the DSS adds only the small, bandwidth-bound
interface traffic of Table~\ref{tab:dss_structured_vs_unstructured}, whereas the assembled BP1 kernel pays the full
uncoalesced gather--scatter cost on every apply.

\begin{table}
    \centering
\begin{tabular}{llccccccc}
    \toprule
    \multicolumn{2}{l}{} & $p=1$  & $p=2$ & $p=3$ & $p=4$ & $p=5$ & $p=6$ & $p=7$         \\
    \midrule
    \multirow{3}{*}{Structured}
                         & GDoF/s & 3.9   & 13.9  & 26.8  & 39.3  & 54.0  & 68.2  & 81.1  \\
                         & GB/s   & 1 686 & 1 623 & 1 605 & 1 536 & 1 561 & 1 561 & 1 538 \\
                         & SoL    & 85 \% & 81 \% & 82 \% & 79 \% & 79 \% & 80\%  & 78\%  \\
    \midrule
    \multirow{3}{*}{Unstructured}
                         & GDoF/s & 3.9   & 12.9  & 21.4  & 29.5  & 38.2  & 48.1  & 59.2  \\
                         & GB/s   & 1 178 & 1 277 & 1 055 & 985   & 943   & 935   & 976   \\
                         & SoL    & 62 \% & 67 \% & 64 \% & 61 \% & 61 \% & 62\%  & 61\%  \\
    \bottomrule
\end{tabular}

    \caption{DSS throughput per unique DoF [GDoF/s] and DRAM speed-of-light (SoL, fraction of peak memory bandwidth) for structured
        (Cartesian, dimension-split path) and unstructured meshes on the unit cube, best block configurations,
        $p = 1, \dots, 7$. Note that DSS is memory-bound, so only memory transfers are reported.}
    \label{tab:dss_structured_vs_unstructured}
\end{table}

\begin{table}
    \centering
\begin{tabular}{llccccccc}
    \toprule
    DSS path                                      & Volume path & $p=1$ & $p=2$ & $p=3$ & $p=4$ & $p=5$ & $p=6$ & $p=7$ \\
    \midrule
    \multirow{2}{*}{Structured}
                                                  & Cartesian   & 3.0   & 9.7   & 15.5  & 19.3  & 19.5  & 21.0  & 27.0  \\
                                                  & General     & 2.7   & 8.0   & 13.9  & 17.3  & 18.4  & 17.7  & ---   \\
    \midrule
    \multirow{2}{*}{Unstructured}
                                                  & Cartesian   & 3.0   & 9.2   & 13.5  & 16.6  & 17.0  & 18.6  & 24.0  \\
                                                  & General     & 2.7   & 7.6   & 12.3  & 15.1  & 16.1  & 16.0  & ---   \\
    \midrule
    \texttt{libCEED}~\cite{osti_1822610}, Fig.~21 &             & 1.69  & 3.74  & 3.93  & 5.05  & 4.19  & 4.92  &       \\
    \bottomrule
\end{tabular}


    \caption{Throughput per unique DoF [GDoF/s] of the \emph{entire} mass operator (volumetric apply followed by DSS),
        $p = 1, \dots, 7$, on a $128^3$ mesh (A100 80\,GB SXM, fp64). The total cost is $t_\text{vol} + t_\text{DSS}$:
        the volumetric apply depends only on the geometry (Cartesian or general/curved), while the DSS exchange depends
        only on the path (structured or unstructured). Throughput is $N_\text{unique}/(t_\text{vol}+t_\text{DSS})$. The $p=7$ general-geometry entries are omitted (exceeded memory).
        The bottom row gives the libCEED BP1 throughput~\cite{osti_1822610} on the same A100 80\,GB SXM; as BP1 uses a
        general geometry, the general-volume/structured-path row is its closest counterpart.}
    \label{tab:combined_operator}
\end{table}

\subsubsection{DSS on unstructured meshes}
\label{sec:results_dss_unstructured_mesh}

The benchmarks so far ran on Cartesian or single-block geometries. We now exercise the DSS operator on meshes that are
representative of production setups, reporting the DSS cost in isolation (the volumetric apply is excluded); the meshes
are detailed in the captions and we read off the qualitative trends here.

The representative case is Figure~\ref{fig:dss_showcase}: a curved ball domain obtained by $6$ global refinements of an
unstructured coarse grid, with no hanging nodes. Such a mesh is heavily block-structured in the interior of each coarse
cell yet stitched together by genuine unstructured exchanges across the macro-block interfaces, which is the very mix a
real application presents, and one that drives both the structured sweeps and the unstructured inter-block kernels at
once. Its DSS throughput (Figure~\ref{fig:dss_showcase}b) sits within a few percent of the structured-Cartesian ceiling
of Table~\ref{tab:dss_structured_vs_unstructured} across the whole degree range, well above the all-unstructured rate
of the same table. Because the unstructured macro-interfaces form only a sparse skeleton of the mesh, they are slower
per DoF, but touch too few DoFs to move the aggregate. Two facts reinforce each other here: a curved mapping leaves the
assembly untouched (DSS moves only DoF values and references no geometric data, unlike the volumetric apply, whose cost
the per-quadrature-point Jacobians inflate), so the general-mesh penalty of the combined operator lives in the
volumetric stage; and a realistic, predominantly block-structured topology runs near the structured best case rather
than the unstructured worst case.

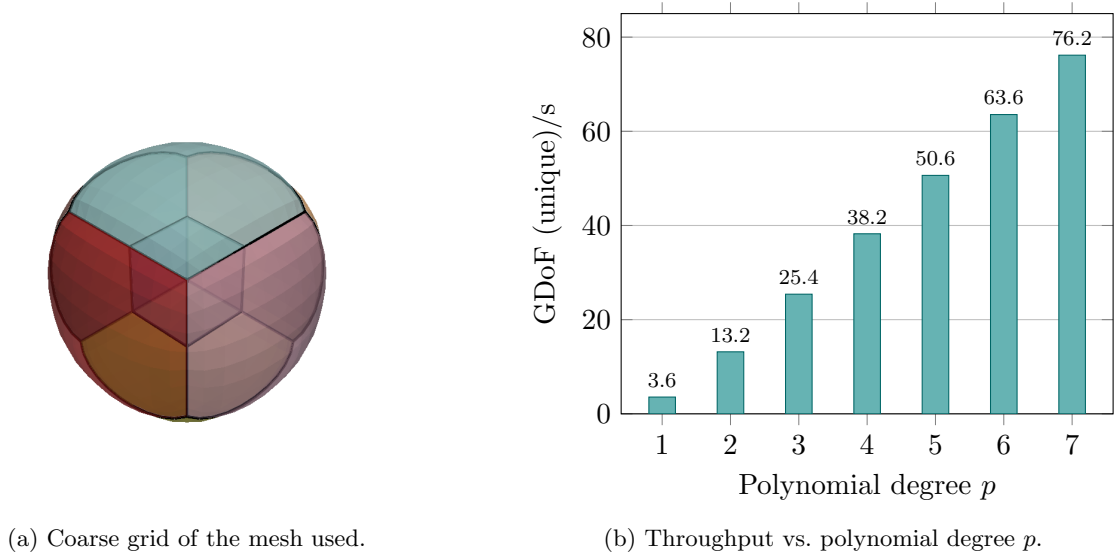
\begin{figure}
    \centering
    \begin{subfigure}[b]{0.49\linewidth}
        \centering
        \includegraphics[width=\linewidth]{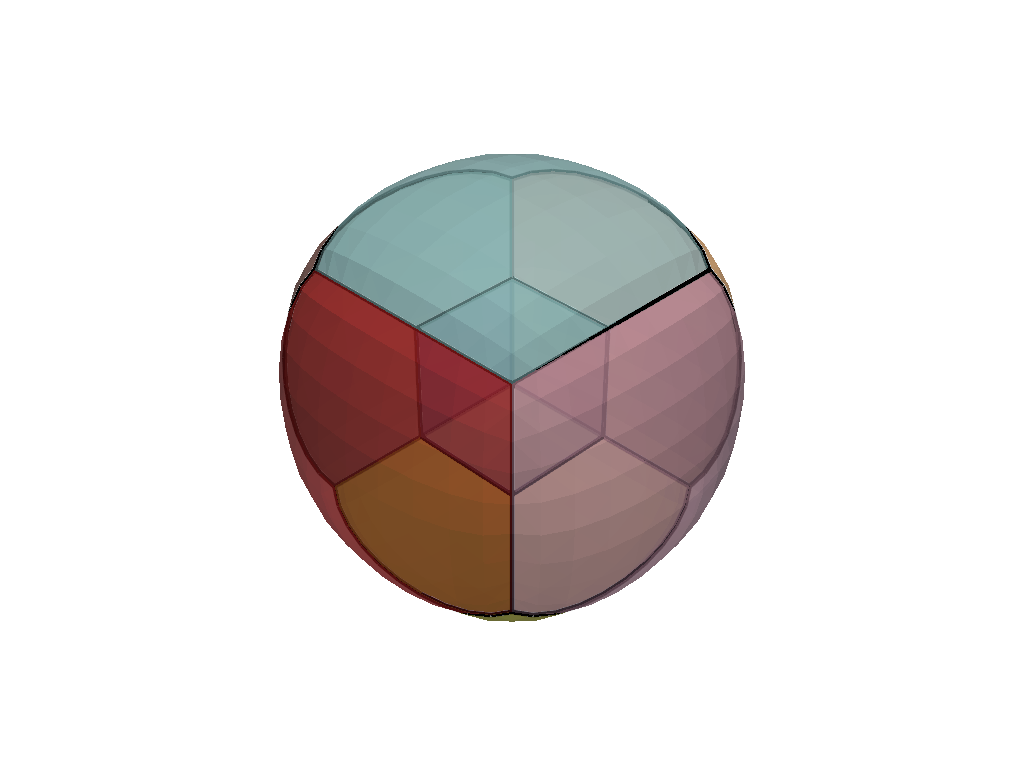}
        \caption{Coarse grid of the mesh used.}
        \label{fig:dss_showcase_mesh}
    \end{subfigure}
    \hfill
    \begin{subfigure}[b]{0.49\linewidth}
        \centering
\begin{tikzpicture}
    \begin{axis}[
            width=\linewidth,
            height=0.85\linewidth,
            ybar,
            bar width=10pt,
            xlabel={Polynomial degree $p$},
            ylabel={ GDoF (unique)/s},
            xtick={1,2,3,4,5,6,7},
            xmin=0.4, xmax=7.6,
            ymin=0, ymax=85,
            enlarge x limits=false,
            ymajorgrids=true,
            nodes near coords,
            nodes near coords style={font=\scriptsize, /pgf/number format/fixed, /pgf/number format/precision=1},
            every node near coord/.append style={anchor=south},
        ]
        \addplot[fill=teal!60, draw=teal!80!black]
        table[col sep=comma, x=p, y=gdofs_uniq] {raw_results/showcase.csv};
    \end{axis}
\end{tikzpicture}
        \caption{Throughput vs.\ polynomial degree $p$.}
        \label{fig:dss_showcase_throughput}
    \end{subfigure}
    \caption{DSS-operator throughput on a curved (ball) geometry with $6$ global refinements, $2\,097\,151$ cells and
        no hanging nodes (A100 80\,GB SXM, fp64). This mesh is representative of a production geometry: predominantly
        block-structured, but stitched together by unstructured exchanges across the macro-block interfaces, so it
        drives the structured and unstructured DSS kernels jointly. (a)~coarse grid of the mesh used; (b)~throughput per
        unique DoF [GDoF/s] as a function of polynomial degree $p$. The measured cell-wise rate (total DoFs divided by the average apply time) is
        rescaled to unique DoFs by the large-mesh factor $\big(p/(p+1)\big)^3$, which neglects the domain boundary and is
        therefore a slight, conservative underestimate of the per-unique throughput.}
    \label{fig:dss_showcase}
\end{figure}

\section{Conclusions}
\label{sec:conclusions}

In this work, the redundant, cell-wise vector is the \emph{persistent} state of a high-order continuous finite element
solver rather than a transient byproduct of matrix-free evaluation, and we have built the machinery that makes this
practical on GPUs. The theoretical core is a primal-dual equivalence theorem: provided the preconditioner returns a
continuous correction, the flexible conjugate gradient iteration executed on unassembled cell-wise data is identical,
iterate by iterate, to the assembled solve, so inter-element communication is confined to a single direct stiffness
summation inside the preconditioner and removed from the Krylov projection altogether. The algorithmic core is a DSS
operator that realizes this summation without indirect addressing, atomics, or coloring: a dimensionally-split cascade
of one-to-one face exchanges that accumulates edge and vertex sums as a byproduct of sequential axis passes, extended
to unstructured macro-interfaces by disjoint topological kernels and to $h$-adaptive meshes by a shadow-cell wrapper
that leaves the high-throughput sweeps untouched. Because cell-wise storage decouples the memory layout from the mesh
topology, the layout itself becomes a tuning parameter, which we mapped across the blocked design space.

The payoff is measured on an A100. The volumetric operators run at the memory roofline: the mass operator transfers
exactly the two doubles per DoF its formulation demands, and the Laplace operator, despite roughly three times the
arithmetic, runs \emph{faster} still, because its contractions fill the tensor cores more completely and leave it even
more firmly memory-bound. The structured DSS sustains around $80\%$ of peak bandwidth, and the entire assembled
operator (volumetric apply plus a full DSS) exceeds the libCEED BP1 benchmark on the same hardware across the whole
degree range, by up to roughly a factor of four, already at the lowest order. On a production-like curved mesh the
assembly runs within a few percent of its structured ceiling, confirming that the unstructured machinery, though slower
per DoF, is exercised on too thin a skeleton to matter. The complete stack is implemented in compact Triton kernels,
runs unmodified on NVIDIA and AMD hardware, and contributed a small-tile tensor-core instruction shape upstream.

The price is memory footprint: the cell-wise format stores up to $8\times$ as many values as the assembled vector at
$p=1$, a redundancy that falls steeply with degree and is mildest precisely in the high-order regime the method
targets. The trade is a larger but coalesced footprint in exchange for eliminating the indirect gather-scatter traffic
that bottlenecks classical pipelines, and the measurements show the exchange to be favourable at every degree tested.

Several directions follow naturally. The flexible Krylov variant was adopted precisely so that the preconditioner may
change between iterations and, in particular, be evaluated in \emph{reduced precision}~\cite{goddeke2007performance}:
because the solver is memory-bound, running the preconditioner in single or half precision directly and substantially
lowers the bytes transferred per DoF (the quantity that governs performance throughout this work) at no cost to the
accuracy of the outer iteration, which remains in double precision. The companion paper~\cite{wichrowski2026MG}
develops the geometric multigrid that exploits this, and shows in addition that DSS across hanging nodes is not
required by the solver, so the one bounded overhead identified here can be sidestepped. The framework is moreover
agnostic to the smoother: overlapping patch smoothers of Schwarz type satisfy the consistency condition of the
equivalence theorem by construction (Remark~\ref{rem:beyond_dss}), each patch solve returning a continuous correction,
so used as the preconditioner they render the outer iteration communication-free \emph{without any explicit DSS at
    all}. The same agnosticism extends to the operator (any symmetric operator fits the theorem unchanged) opening the
communication-free treatment to a far broader class of problems than the mass and Laplace operators benchmarked here.

\paragraph{Acknowledgements}
The author would like to thank Mario Lezcano Casado (OpenAI) and Keren Zhou (George Mason University) for reviewing and
accepting the author's Triton patch enabling \texttt{fp64} MMA support ($8{\times}8{\times}4$ tensor-core
instructions), on which the kernels of this work rely. Mario Lezcano Casado additionally fixed an independent Triton
bug causing shared-memory bank conflicts for 64-bit operands, which substantially improved the throughput of the
higher-degree kernels ($p = 4, 5$).


\paragraph{Declarations}
Language models (Claude, Gemini) were used to produce unassembled initial drafts and to smooth the assembled text; as
everywhere in this framework, assembly, and full accountability for all scientific content, remains with the author.

\FloatBarrier
\bibliographystyle{siam}
\bibliography{literature, added_literature}

\newpage

\appendix

\section{Extension of DSS algorithms onto Non-Conforming Interfaces}
\label{sec:dss_hanging}

The construction below extends the DSS sweeps to $h$-adaptive meshes with hanging nodes. It currently works on
structured (Cartesian) refinement; the unstructured hanging-node case is not yet robust, and the companion multigrid
paper~\cite{wichrowski2026MG} shows that the solver does not require DSS across hanging nodes at all, so this path is
not needed for the results of the present work. We record it here for completeness.

In the context of $h$-adaptive meshes with hanging nodes, interfaces may bridge elements at different hierarchical
levels. To process these non-conforming boundaries without disrupting the highly optimized structured and unstructured
kernels, we deploy a pre- and post-processing wrapper relying on \emph{Shadow Cells}. A shadow cell is a structurally
identical, conceptually dormant parent cell that is allocated in the cell-wise storage exactly like an active cell but
exists exclusively to broker data across a refinement interface (Figure~\ref{fig:dss_shadow}). Rather than treating
hanging nodes as a special case during the summation sweeps, the algorithm restricts fine-level data to the
coarse-level shadows, runs the standard structured and unstructured DSS on a topology that is globally conforming at
the coarse level, and then prolongates the assembled data back to the fine cells. The non-conforming case is thereby
reduced \emph{exactly} to the conforming one, and the summation kernels remain completely unaware of adaptivity.

\begin{figure}[htbp]
    \centering
\begin{tikzpicture}[
                scale=0.85,
                coarse/.style={draw=black, thin},
                shadow/.style={draw=black, thin, fill=gray!45},
                fine/.style={draw=black!70, line width=0.4pt},
                redge/.style={red, line width=1.6pt, shorten >=2.5pt, shorten <=2.5pt}, ]
        \def\refined{%
                0/0,1/0,2/0,3/0,4/0,5/0,6/0,7/0,%
                0/1,1/1,2/1,3/1,4/1,5/1,6/1,7/1,%
                2/2,3/2,4/2,5/2}
        \def\shadows{%
                0/1,1/1,2/1,5/1,6/1,7/1,%
                2/2,3/2,4/2,5/2}

        \draw[coarse] (0,0) rectangle (8,4);
        \foreach \i in {1,...,7} { \draw[coarse] (\i,0) -- (\i,4); }
        \foreach \j in {1,...,3} { \draw[coarse] (0,\j) -- (8,\j); }

        \foreach \c/\r in \shadows {
                \fill[shadow] (\c,\r) rectangle (\c+1,\r+1);
        }

        \foreach \c/\r in \refined {
                \draw[coarse] (\c,\r) rectangle (\c+1,\r+1);
                \draw[fine] (\c+0.5,\r) -- (\c+0.5,\r+1);   
                \draw[fine] (\c,\r+0.5) -- (\c+1,\r+0.5);   
        }

        \foreach \x in {0,1,6,7} { \draw[redge] (\x,2) -- (\x+1,2); }   
        \foreach \x in {2,3,4,5} { \draw[redge] (\x,3) -- (\x+1,3); }   
        \draw[redge] (2,2) -- (2,3);                                    
        \draw[redge] (6,2) -- (6,3);                                    

        \fill[red] (2.15,1.85) circle (2.2pt);   
        \fill[red] (5.85,1.85) circle (2.2pt);   

\end{tikzpicture}
    \caption{Shadow cells on a 2D $h$-adaptive mesh. Every coarse cell that touches the refinement interface is
        materialized in cell-wise storage as a shadow: the dormant parent of the fine children that lie across the
        interface (shaded gray). Cells in the interior of a refined patch, which border no coarse neighbour, need no
        shadow and are left unshaded. The red entries are the auxiliary degrees of freedom introduced solely to broker
        the DSS exchange across the interface: the thick red segments are the shadow faces along the interface (drawn
        broken to indicate that each coarse face carries its own copy, separated at the coarse grid nodes), and the red
        dots are the auxiliary vertex DoFs at the re-entrant step corners. Fine-level data are restricted onto these
        shadow entries, the standard conforming DSS of
        Algorithms~\ref{alg:struct_dss}--\ref{alg:unstruct_dss} runs on the resulting coarse-conforming topology, and the
        assembled values are prolongated back to the fine cells.}
    \label{fig:dss_shadow}
\end{figure}
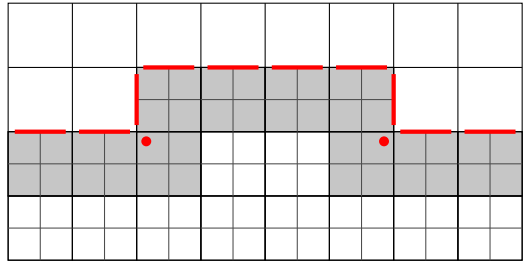

The pre-processing step executes a \emph{Restrict-Add} and \emph{Zero-Out} pattern. First, all shadow cells are
explicitly zeroed to clear any potential stale data. Then, the dual residual degrees of freedom from the fine cells
lying on the refinement edge are restricted to the coarse shadow cell using 1D tensor-product restriction operators
$\mathcal{R}_h$ and accumulated. Immediately following this accumulation, the corresponding values on the fine children
are zeroed out. This zeroing process is critical to prevent double-counting when the subsequent standard DSS passes
inevitably process the perpendicular axes.

Once the pre-processing is complete, the standard structured and unstructured DSS operations are executed, naturally
including the populated shadow cells as valid participants in the global summation. In the post-processing phase, the
assembled data in the shadow cells is lifted back to the fine grid. We inject the coarse shadow values back into the
fine children using the prolongation operators $\mathcal{P}_h$ and accumulate them. To prevent the injection from
corrupting vertex values due to interpolation coupling, a save-restore pattern is used for shadow vertices before the
face injection. Finally, the fine cells are zeroed out locally before receiving the accumulated values, and the shadow
cells themselves are reset. This completely encapsulates the adaptivity logic outside of the high-throughput summation
sweeps.

\begin{algorithm}
    \caption{Full DSS Operator with Hanging Node Resolution}
    \label{alg:hanging_dss}
    \begin{algorithmic}[1]
        \Require Cell-wise residual vector $u$, fine refinement cells $\mathcal{T}_{\text{fine}}$, coarse shadow cells $\mathcal{T}_{\text{shadow}}$.
        \Ensure Globally assembled continuous primal residual field $u$.

        \State \textbf{Pre-processing (Restriction):}
        \State $u|_{\mathcal{T}_{\text{shadow}}} \gets 0$ \Comment{Clear dormant shadows}
        \For{\textbf{each} shadow cell $K_S \in \mathcal{T}_{\text{shadow}}$ and its fine children $K_F \in \text{children}(K_S)$}
        \State \text{Restrict-Add:} $u|_{K_S} \gets u|_{K_S} + \mathcal{R}_h(u|_{K_F})$
        \State \text{Zero-Out:} $u|_{K_F} \gets 0$ \Comment{Prevent double counting}
        \EndFor

        \State \textbf{Core Assembly:}
        \State $u \gets \textsc{Alg~1~Structured~Summation}(u)$
        \State $u \gets \textsc{Alg~2~Unstructured~Summation}(u)$

        \State \textbf{Post-processing (Prolongation):}
        \State $v_{\text{tmp}} \gets u|_{\mathcal{T}_{\text{shadow}}, \text{vertices}}$ \Comment{Save vertex values to prevent injection corruption}
        \For{\textbf{each} shadow cell $K_S \in \mathcal{T}_{\text{shadow}}$ and its fine children $K_F \in \text{children}(K_S)$}
        \State $u|_{K_F} \gets 0$ \Comment{Clear existing values}
        \State \text{Prolong-Add:} $u|_{K_F} \gets u|_{K_F} + \mathcal{P}_h(u|_{K_S})$
        \EndFor
        \State $u|_{\mathcal{T}_{\text{fine}}, \text{vertices}} \gets v_{\text{tmp}}$ \Comment{Restore uncontaminated fine vertex values}
        \State $u|_{\mathcal{T}_{\text{shadow}}} \gets 0$ \Comment{Reset shadow cells}

        \State \Return $u$
    \end{algorithmic}
\end{algorithm}

On a structured $h$-adaptively refined ball ($5$ global plus $4$ adaptive refinements, $1.2\times10^{5}$ hanging
interfaces) the wrapper's throughput is reported in Table~\ref{tab:dss_adaptive}. The wrapper's restriction and
prolongation add work that grows as $(p{+}1)^2$, so the rate trails the conforming case---by a few tens of percent at
low order, up to roughly a factor of two-and-a-half by $p=5$---yet it stays bounded and keeps rising with $p$, since
the non-conforming problem is reduced \emph{exactly} to the conforming one and adaptivity never touches the
high-throughput sweeps.

\begin{table}
    \centering
\begin{tabular}{lccccc}
    \toprule
    $p$                 & 1   & 2   & 3    & 4    & 5    \\
    \midrule
    Throughput [GDoF/s] & 2.8 & 9.6 & 14.2 & 17.4 & 18.7 \\
    \bottomrule
\end{tabular}

    \caption{DSS operator throughput on a locally refined (i.e. with hanging-nodes) mesh as a function of polynomial degree
        $p$. The mesh is built from $5$ global refinements plus $4$ adaptive refinements in the shell $0.2 < r < 0.5$
        ($9\,078\,473$ cells, $120\,216$ hanging interfaces; A100 80\,GB SXM, fp64). Throughput is reported per unique
        DoF: the measured cell-wise rate (total DoFs divided by the average apply time) is rescaled by the large-mesh
        factor $\big(p/(p+1)\big)^3$, which neglects the domain boundary and is therefore a slight, conservative
        underestimate.}
    \label{tab:dss_adaptive}
\end{table}

\end{document}